\documentclass[11pt]{amsart}
\usepackage{amscd}
\usepackage{amssymb}
\usepackage{graphicx}
\newcounter{mtheorem}

\newtheorem{theorem}{Theorem}[section]
\newtheorem{lemma}[theorem]{Lemma}
\newtheorem{prop}[theorem]{Proposition}
\newtheorem{corollary}[theorem]{Corollary}

\theoremstyle{definition}
\newtheorem{definition}[theorem]{Definition}
\newtheorem{example}[theorem]{Example}
\newtheorem{decomp}{Decomposition}
\theoremstyle{remark}
\newtheorem{remark}[theorem]{Remark}

\numberwithin{equation}{section}
\newcommand{\E}{\mathcal{E}}
\newcommand{\cone}{\mathcal{C}}

\newcommand{\unitary}[1]{\textrm{U({#1})}}
\newcommand{\sunitary}[1]{\textrm{SU({#1})}}
\newcommand{\sorth}[1]{\textrm{SO({#1})}}

\newcommand{\C}{\mathbb{C}}
\newcommand{\R}{\mathbb{R}}

\newcommand{\N}{\mathbb{N}}

\newcommand{\Sph}{\mathbb{S}}

\newcommand{\Imag}{\operatorname{Im}}
\newcommand{\Real}{\operatorname{Re}}

\newcommand{\tnabla}{{\widetilde{\nabla}}}
\newcommand{\tg}{{\tilde{g}}}
\newcommand{\tD}{\widetilde{D}}
\newcommand{\tGamma}{{\widetilde{\Gamma}}}
\title[SL conifolds, I]{Special Lagrangian conifolds, I: Moduli spaces \linebreak(extended version)}
\author[T.~Pacini]{Tommaso~Pacini}
\address{Scuola Normale Superiore, Pisa} \email{tommaso.pacini@sns.it}
\date{\today}
\begin{document}
\begin{abstract}
This is the extended version of the paper \cite{pacini:sldefs}, which discusses the deformation theory of special Lagrangian (SL) conifolds in $\C^m$. Conifolds are a key ingredient in the compactification problem for moduli spaces of compact SLs in Calabi-Yau manifolds. The conifold category
allows for the simultaneous presence of conical singularities and of non-compact, asymptotically conical, ends. 

Our main theorem is the natural next step in the chain of results
initiated by McLean \cite{mclean} and continued by the author \cite{pacini:defs}
and Joyce \cite{joyce:II}. We survey all these results, providing a unified framework for studying the
various 
cases and emphasizing analogies and differences. Compared to \cite{pacini:sldefs}, this paper contains more detail but the same results. The paper also lays down the geometric foundations for
our paper \cite{pacini:slgluing} concerning gluing constructions for SL conifolds in $\C^m$. 
\end{abstract}
\maketitle
\tableofcontents
\section{Introduction}\label{s:intro}
Let $M$ be a Calabi-Yau (CY) manifold. Roughly speaking, a submanifold
$L\subset M$ is \textit{special Lagrangian} (SL) if it is both minimal and
Lagrangian with respect to the ambient Riemannian and symplectic structures. 

From the point of view of Riemannian Geometry it is of course natural to focus
on the minimality condition. It turns out that SLs are automatically
volume-minimizing in their homology class. In fact, this was Harvey and Lawson's
main motivation for defining and studying SLs within the general context of
Calibrated Geometry \cite{harveylawson}. This is still the most common point of
view on SLs and leads to emphasizing the role of analytic and Geometric Measure
Theory techniques. It also provides a connection with various classical problems
in Analysis such as the Plateau problem and the study of area-minimizing cones.
In many ways it is the point of view adopted here.

From the point of view of Symplectic Geometry it is instead natural to focus on
the Lagrangian condition. Specifically, SLs are examples of Maslov-zero
Lagrangian submanifolds. This leads to emphasizing the role of Symplectic
Topology techniques, both classical (such as the h-principle and moment maps)
and contemporary (such as Floer homology). An early instance of this point of
view is the work of Audin \cite{audin}; it also permeates the paper
\cite{haskinspacini} by Haskins and the author.

Given this richness of ingredients it is perhaps not surprising that SLs are
conjectured to play an important role in Mirror Symmetry \cite{kontsevich},
\cite{syz} and to produce interesting new invariants of CY manifolds
\cite{joyce:3spheres}. Likewise, and more intrinsically, they also tend to
exhibit other nice technical features. In particular it is by now well
understood that SLs often generate smooth, finite-dimensional, moduli spaces.
This SL deformation problem has been studied by a number of authors under
various topological and geometric assumptions. One clear path is the chain of
results initiated by McLean \cite{mclean}, who studied deformations of smooth
compact SLs; continued by the author \cite{pacini:defs} and Marshall
\cite{marshall}, who adapted that set-up to study certain smooth non-compact
(\textit{asymptotically conical}, AC) SLs; and further advanced by Joyce, who
presented analogous results for compact \textit{conically singular} (CS) SLs
\cite{joyce:II}. 

The above three classes of SLs are intimately linked, as follows. One of the main open questions in SL geometry is how to compactify McLean's moduli spaces. This problem is currently one of the biggest obstructions to progress on the above conjectures. Roughly speaking, compactifying the moduli space requires adding to it a ``boundary'' containing singular compact SLs. By definition, CS SLs have isolated singularities modelled on SL cones in $\C^m$:  they would be the simplest objects appearing in this boundary. If a CS SL appears in the boundary, it must be a limit of a 1-parameter family of smooth compact SLs. These smooth SLs can be recovered via a gluing construction which desingularizes the CS SL: (i) each singularity of the CS SL defines a SL cone in $\C^m$; (ii) each of these cones must admit a 1-parameter family of SL desingularizations, \textit{i.e.} AC SLs in $\C^m$ converging to the cone as the parameter $t$ tends to $0$; (iii) the family of smooth SLs is obtained by gluing the AC SLs into a neighbourhood of the singularities of the CS SL. This picture is made precise by Joyce's gluing results \cite{joyce:III}, \cite{joyce:IV},
\cite{joyce:V}. Section 8 of \cite{joyce:V} then shows that, in some cases and near the boundary, 
the compactified moduli space can be locally written as a product of moduli
spaces of AC and CS SLs.

The above classes of submanifolds are special cases within the broader
category of \textit{Riemannian conifolds}, which includes manifolds
exhibiting both AC and CS ends. In other words, it allows CS SLs to become non-compact by allowing the presence of AC ends. This is of fundamental importance for the construction of SLs in $\C^m$: it is well-known that $\C^m$ does not admit any compact (smooth or singular) volume-minimizing submanifolds. Cones in $\C^m$ with an isolated singularity at
the origin are the simplest example of conifold: the construction of new examples and the study of their properties is currently one of the most active areas of SL research \cite{harveylawson},
\cite{haskins}, \cite{haskinskapouleas}, \cite{haskinspacini}, \cite{joyce:symmetries}, \cite{ohnita}. Conifolds provide the appropriate framework in which to extend all the above research. In particular, they might also substitute AC SLs in Joyce's gluing results: one could try to cut out a conical singularity of the CS SL and replace it with a different singular conifold, thus jumping from one area of the boundary of the compactified moduli space, containing certain CS SLs, to another. 

The paper at hand is Part I of a multi-step project aiming to set up a general theory of SL conifolds. Two other papers related to this project are currently available: \cite{pacini:weighted}, \cite{pacini:slgluing} (see also \cite{pacini:sldefs}). Further work is in progress. The goal of this paper is to provide a general deformation theory of SL conifolds in $\C^m$. The best set-up for the SL deformation
problem is the one provided by Joyce \cite{joyce:II}. It is based on his
Lagrangian neighbourhood and regularity theorems \cite{joyce:I}. Joyce's
framework has two benefits: (i) it simplifies the Analysis via a reduction from
the semi-elliptic operator $d\oplus d^*$ on 1-forms to the
elliptic Laplace operator on functions, (ii) it nicely emphasizes the separate
contributions to the dimension of $\mathcal{M}_L$ coming from the topological
and from the analytic components. Along with the main result Theorem
\ref{thm:accssl} concerning moduli spaces of CS/AC SL submanifolds in $\C^m$, we thus
present new proofs of the previously-known results, emphasizing this point of view. 
In this sense, this paper also serves the purpose of surveying and unifying those results. More importantly, it lays down the geometric foundations for \cite{pacini:slgluing}; the analytic foundations are provided by \cite{pacini:weighted}.

We now summarize the contents of this paper. Section \ref{s:geometry_review}
introduces the category of $m$-dimensional Riemannian conifolds. The main
definitions are standard but Section \ref{ss:closedforms} contains an
investigation into the structure of various spaces of closed 1-forms on these
manifolds. This is a fundamental component of the Lagrangian and SL deformation
theory. The corresponding notion of ``subconifolds'' is presented in Section
\ref{s:lagconifolds}, leading to the concept of \textit{Lagrangian conifolds}.
Deformation theory begins in Section \ref{s:lagdefs}. From various points of
view it seems most satisfying to begin with the general (infinite-dimensional)
theory of Lagrangian deformations. This is presented as a direct consequence of
Joyce's Lagrangian neighbourhood theorems, coupled with the material of Section
\ref{ss:closedforms}. The case of Lagrangian cones is studied in particular
detail in Section \ref{ss:conelagdefs} as it provides the backbone for all other
cases. After presenting the necessary definitions in Section \ref{s:slgeometry},
the analogous framework for deforming SL conifolds is developed in Section
\ref{s:setup}. With the aim of making this paper reasonably self-contained,
Section \ref{s:reviewlaplace} summarizes from \cite{pacini:weighted} some
results concerning harmonic functions on conifolds. The SL deformation theory is
then completed in Section \ref{s:moduli}. The proofs rely upon a fair amount of
analytic machinery: weighted Sobolev spaces, embedding theorems and the theory
of elliptic operators on conifolds. Full details are provided in \cite{pacini:weighted}.

\ 

\textbf{Important remark:} To simplify certain arguments, throughout this paper
we assume $m\geq 3$.


\section{Geometry of conifolds}\label{s:geometry_review}

\subsection{Asymptotically conical and conically singular
manifolds}\label{ss:accs_review}

We introduce here the categories of differentiable and Riemannian manifolds
mainly relevant to this paper, referring to \cite{pacini:weighted} for further
details. Following \cite{joyce:I}, however, we introduce a small variation of
the notion of ``conically singular" manifolds:  presenting them in terms of the
compactification $\bar{L}$ will allow us to keep track of the singular points
$x_i$. This plays no role in this section but in Section \ref{s:lagdefs} it will
become very useful.

\begin{definition}\label{def:manifold_ends}
Let $L^m$ be a smooth manifold. We say $L$ is a \textit{manifold with ends} if
it satisfies the following conditions:
\begin{enumerate}
\item We are given a compact subset $K\subset L$ such that $S:=L\setminus K$ has
a finite number of connected components $S_1,\dots,S_e$, \textit{i.e.}
$S=\amalg_{i=1}^e S_i$.
\item For each $S_i$ we are given a connected ($m-1$)-dimensional compact
manifold $\Sigma_i$ without boundary. 
\item There exist diffeomorphisms $\phi_i:\Sigma_i\times [1,\infty)\rightarrow
\overline{S_i}$.
\end{enumerate}
We then call the components $S_i$ the \textit{ends} of $L$ and the manifolds
$\Sigma_i$ the \textit{links} of $L$. We denote by $S$ the union of the ends and
by $\Sigma$ the union of the links of $L$. 
\end{definition}

\begin{definition}\label{def:metrics_ends}
Let L be a manifold with ends. Let $g$ be a Riemannian metric on $L$. Choose an
end $S_i$ with corresponding link $\Sigma_i$.

We say that $S_i$ is a \textit{conically singular} (CS) end if the following
conditions hold:
\begin{enumerate}
\item $\Sigma_i$ is endowed with a Riemannian metric $g_i'$.

We then let $(\theta,r)$ denote the generic point on the product manifold
$C_i:=\Sigma_i\times (0,\infty)$ and $\tg_i:=dr^2+r^2g_i'$ denote the
corresponding \textit{conical metric} on $C_i$.
\item There exist a constant $\nu_i>0$ and a diffeomorphism
$\phi_i:\Sigma_i\times (0,\epsilon]\rightarrow \overline{S_i}$ such that, as $r\rightarrow
0$ and for all $k\geq 0$,
$$|\tnabla^k(\phi_i^*g-\tg_i)|_{\tg_i}=O(r^{\nu_i-k}),$$
where $\tnabla$ is the Levi-Civita connection on $C_i$ defined by $\tg_i$. 
\end{enumerate}
We say that $S_i$ is an \textit{asymptotically conical} (AC) end if the
following conditions hold:
\begin{enumerate}
\item $\Sigma_i$ is endowed with a Riemannian metric $g_i'$.

We again let $(\theta,r)$ denote the generic point on the product manifold
$C_i:=\Sigma_i\times (0,\infty)$ and $\tg_i:=dr^2+r^2g_i'$ denote the
corresponding conical metric on $C_i$.
\item There exist a constant $\nu_i<0$ and a diffeomorphism
$\phi_i:\Sigma_i\times [R,\infty)\rightarrow \overline{S_i}$ such that, as $r\rightarrow
\infty$ and for all $k\geq 0$,
$$|\tnabla^k(\phi_i^*g-\tg_i)|_{\tg_i}=O(r^{\nu_i-k}),$$
where $\tnabla$ is the Levi-Civita connection on $C_i$ defined by $\tg_i$.
\end{enumerate}
In either of the above situations we call $\nu_i$ the \textit{convergence rate}
of $S_i$.
\end{definition}

We refer to \cite{pacini:weighted} Section 6 for a better understanding of the asymptotic
conditions introduced in Definition \ref{def:metrics_ends}.

\begin{definition}\label{def:cs_manifold}
Let $(\bar{L},d)$ be a metric space. $\bar{L}$ is a \textit{Riemannian manifold
with conical singularities} (CS manifold) if it satisfies the following
conditions.
\begin{enumerate}
\item We are given a finite number of points $\{x_1,\dots,x_e\}\in \bar{L}$ such that $L:=\bar{L}\setminus\{x_1,\dots,x_e\}$ has the
structure of a smooth $m$-dimensional manifold with $e$ ends. 

More specifically, we assume given $\epsilon\in (0,1)$ such that any pair of
distinct points satisfies $d(x_i,x_j)>2\epsilon$. Set $S_i:=\{x\in L:
0<d(x,x_i)<\epsilon\}$. We then assume that $S_i$ are the ends of $L$ with
respect to some given connected links $\Sigma_i$. 
\item We are given a Riemannian metric $g$ on $L$ inducing the distance $d$. 
\item With respect to $g$, each end $S_i$ is CS in the sense of Definition
\ref{def:metrics_ends}.
\end{enumerate}
It follows from our definition that any CS manifold $\bar{L}$ is compact. We
will often not distinguish between $\bar{L}$ and $L$, but notice that $(L,g)$ is
neither compact nor complete. We call $x_i$ the \textit{singularities} of
$\bar{L}$. 
\end{definition}

\begin{definition}\label{def:ac_manifold}
Let $(L,g)$ be a Riemannian manifold. $L$ is a \textit{Riemannian manifold with
asymptotically conical ends} (AC manifold) if it satisfies the following
conditions.
\begin{enumerate}
\item $L$ is a smooth manifold with $e$ ends $S_i$ and connected links
$\Sigma_i$.
\item Each end $S_i$ is AC in the sense of Definition \ref{def:metrics_ends}.
\end{enumerate}
\end{definition}
One can check that AC manifolds are non-compact but complete.

\begin{definition} \label{def:accs_manifold}
Let $(\bar{L},d)$ be a metric space. We say that $\bar{L}$ is a
\textit{Riemannian CS/AC manifold} if it satisfies the following conditions. 
\begin{enumerate}
\item We are given a finite number of points $\{x_1,\ldots,x_s\}$ and a number $l$ such that
$L:=\bar{L}\setminus\{x_1,\dots,x_s\}$ has the structure of a smooth
$m$-dimensional manifold with $s+l$ ends. 
\item We are given a metric $g$ on $L$ inducing the distance $d$.
\item With respect to $g$, neighbourhoods of the points $x_i$ have the structure
of CS ends in the sense of Definition \ref{def:metrics_ends}. These are the
``small" ends. We also assume that the remaining ends are ``large",
\textit{i.e.} they have the structure of AC ends in the sense of Definition
\ref{def:metrics_ends}.
\end{enumerate}
We will denote the union of the CS links (respectively, of the CS ends) by
$\Sigma_0$ (respectively, $S_0$) and those corresponding to the AC links and
ends by $\Sigma_\infty$, $S_\infty$. 
\end{definition}

\begin{definition} \label{def:conifold}
We use the generic term \textit{conifold} to indicate any CS, AC or CS/AC
manifold. If $(L,g)$ is a conifold and $C:=\amalg C_i$ is the union of the
corresponding cones as in Definition \ref{def:metrics_ends}, endowed with the
induced metric $\tg$, we say that $(L,g)$ is \textit{asymptotic} to $(C,\tg)$.
\end{definition}

\begin{remark}\label{rem:nondistinct_sings}
 If we think of $\bar{L}$ as a generic compactification of the manifold with ends $L$, we should allow several CS ends to become connected by the addition of a single singular point. Notice however that we have imposed that our links be connected. We should thus allow that our points $x_i$ be not necessarily distinct. This apparent detail becomes extremely relevant when working with ``parametric connect sums'', as in \cite{pacini:weighted}, \cite{pacini:slgluing}. In \cite{pacini:weighted}, however, we do not need to mention it because there the connect sum $L_t$ is defined in terms of $L$: in some sense, the compactification $\bar{L}$ appears only \textit{a posteriori} with respect to the connect sum, as the limit of $L_t$ as $t\rightarrow 0$. In \cite{pacini:slgluing} we again do not need to mention it, this time because the connect sum is defined in terms of an immersion: by definition, the immersion is allowed to identify points so we might as well assume that the $x_i$ and cones are initially distinct. The connect sum then depends only on the identifications determined by the immersion.
\end{remark}

Cones in $\R^n$ are of course the archetype of CS/AC manifold, as follows. 
 
\begin{definition} \label{def:cone}
A subset $\bar{\mathcal{C}}\subseteq\R^n$ is a \textit{cone} if it is invariant
under dilations of $\R^n$, \textit{i.e.} if $t\cdot \bar{\mathcal{C}}\subseteq
\bar{\mathcal{C}}$, for all $t\geq 0$. It is uniquely identified by its
\textit{link} $\Sigma:=\bar{\mathcal{C}}\bigcap \Sph^{n-1}$. We will set
$\mathcal{C}:=\bar{\mathcal{C}}\setminus 0$. The cone is \textit{regular} if
$\Sigma$ is smooth. From now on we will always assume this.

Let $g'$ denote the induced metric on $\Sigma$. Then $\mathcal{C}$ with its
induced metric is isometric to $\Sigma\times (0,\infty)$ with the conical metric
$\tg:=dr^2+r^2g'$. In particular $\bar{\mathcal{C}}$ is a CS/AC manifold; it has
as many AC and CS ends as the number of connected components $\Sigma_i$ of
$\Sigma$. Each $\Sigma_i$ thus defines a singular point $x_i$ but these singular
points are not distinct: they all coincide with the origin. Notice that $\Sigma$
is a subsphere $\Sph^{m-1}\subseteq \Sph^{n-1}$ iff $\bar{\mathcal{C}}$ is an
$m$-plane in $\R^n$.
\end{definition}

Let $E$ be a vector bundle over $(L,g)$. Assume $E$ is endowed with a metric and
metric connection $\nabla$: we say that $(E,\nabla)$ is a \textit{metric pair}.
In later sections $E$ will usually be a bundle of differential forms $\Lambda^r$
on $L$, endowed with the metric and Levi-Civita connection induced from $g$. We
can define two types of Banach spaces of sections of $E$, 
referring to \cite{pacini:weighted} for further details regarding the structure
and properties of these spaces.

Regarding notation, given a vector
$\boldsymbol{\beta}=(\beta_1,\dots,\beta_e)\in \R^e$ and $j\in\N$ we set
$\boldsymbol{\beta}+j:=(\beta_1+j,\dots,\beta_e+j)$. We write
$\boldsymbol{\beta}\geq\boldsymbol{\beta}'$ iff $\beta_i\geq\beta_i'$.

\begin{definition}\label{def:csac_sectionspaces}
Let $(L,g)$ be a conifold with $e$ ends. We say that a smooth function
$\rho:L\rightarrow (0,\infty)$ is a \textit{radius function} if $\rho(x)\equiv
r$ on each end. Given any vector
$\boldsymbol{\beta}=(\beta_1,\dots,\beta_{e})\in\R^{e}$, choose a function
$\boldsymbol{\beta}$ on $L$ which, on each end $S_i$, restricts to the constant
$\beta_i$. 

Given any metric pair $(E,\nabla)$, the \textit{weighted Sobolev spaces} are
defined by
\begin{equation}\label{eq:weighted_sob}
W^p_{k;\boldsymbol{\beta}}(E):=\mbox{Banach space completion of the space
}\{\sigma\in C^\infty(E):\|\sigma\|_{W^p_{k;\boldsymbol{\beta}}}<\infty\},
\end{equation}
where we use the norm
$\|\sigma\|_{W^p_{k;\boldsymbol{\beta}}}:=(\Sigma_{j=0}^k\int_L|\rho^{
-\boldsymbol{\beta}+j}\nabla^j\sigma|^p\rho^{-m}\,\mbox{vol}_g)^{1/p}$. 

The \textit{weighted spaces of $C^k$ sections} are defined by
\begin{equation}\label{eq:weighted_C^k}
C^k_{\boldsymbol{\beta}}(E):=\{\sigma\in C^k(E):
\|\sigma\|_{C^k_{\boldsymbol{\beta}}}<\infty\},
\end{equation}
where we use the norm $\|\sigma\|_{C^k_{\boldsymbol{\beta}}}:=\sum_{j=0}^k
\mbox{sup}_{x\in L}|\rho^{-\boldsymbol{\beta}+j}\nabla^j\sigma|$. Equivalently,
$C^k_{\boldsymbol{\beta}}(E)$ is the space of sections $\sigma\in C^k(E)$ such
that $|\nabla^j \sigma|=O(r^{\boldsymbol{\beta}-j})$ as $r\rightarrow 0$
(respectively, $r\rightarrow\infty$) along each CS (respectively, AC) end. These
are also Banach spaces.

To conclude, the  \textit{weighted space of smooth sections} is defined by
\begin{equation*}
C^\infty_{\boldsymbol{\beta}}(E):=\bigcap_{k\geq 0} C^k_{\boldsymbol{\beta}}(E).
\end{equation*}
Equivalently, this is the space of smooth sections such that $|\nabla^j
\sigma|=O(\rho^{\boldsymbol{\beta}-j})$ for all $j\geq 0$. This space has a natural
Fr\'echet structure. 

When $E$ is the trivial $\R$ bundle over $L$ we obtain weighted spaces of
functions on $L$. We usually denote these by $W^p_{k,\boldsymbol{\beta}}(L)$ and
$C^k_{\boldsymbol{\beta}}(L)$. In the case of a CS/AC manifold we will often
separate the CS and AC weights, writing
$\boldsymbol{\beta}=(\boldsymbol{\mu},\boldsymbol{\lambda})$ for some
$\boldsymbol{\mu}\in \R^s$ and some $\boldsymbol{\lambda}\in \R^l$. We then
write $C^k_{(\boldsymbol{\mu},\boldsymbol{\lambda})}(E)$ and
$W^p_{k,(\boldsymbol{\mu},\boldsymbol{\lambda})}(E)$.
\end{definition}

For these spaces one can prove the validity of the following weighted version of
the Sobolev Embedding Theorems, cf. \cite{pacini:weighted}.

\begin{theorem}\label{thm:embedding}
Let $(L,g)$ be an AC manifold. Let $(E,\nabla)$ be a metric pair over $L$.
Assume $k\geq 0$, $l\in\{1,2,\dots\}$ and $p\geq 1$. Set
$p^*_l:=\frac{mp}{m-lp}$. Then, for all
$\boldsymbol{\beta}'\geq\boldsymbol{\beta}$,
\begin{enumerate}
\item If $lp<m$ then there exists a continuous embedding
$W^p_{k+l,\boldsymbol{\beta}}(E)\hookrightarrow
W^{p^*_l}_{k,\boldsymbol{\beta}'}(E)$.
\item If $lp=m$ then, for all $q\in [p,\infty)$, there exist continuous
embeddings $W^p_{k+l,\boldsymbol{\beta}}(E)\hookrightarrow
W^q_{k,\boldsymbol{\beta}'}(E)$.
\item If $lp>m$ then there exists a continuous embedding
$W^p_{k+l,\boldsymbol{\beta}}(E)\hookrightarrow C^k_{\boldsymbol{\beta}'}(E)$. 
\end{enumerate}
Furthermore, assume $kp>m$. Then the corresponding weighted
Sobolev spaces are closed under multiplication, in the following sense. For
any $\boldsymbol{\beta}_1$ and $\boldsymbol{\beta_2}$ there exists $C>0$ such
that, for all $u\in W^p_{k,\boldsymbol{\beta_1}}$ and $v\in
W^p_{k,\boldsymbol{\beta_2}}$,
\begin{equation*}
\|uv\|_{W^p_{k,\boldsymbol{\beta_1}+\boldsymbol{\beta_2}}}\leq
C\|u\|_{W^p_{k,\boldsymbol{\beta_1}}}\|v\|_{W^p_{k,\boldsymbol{\beta_2}}}.
\end{equation*}
Let $(L,g)$ be a CS manifold. Then the same conclusions hold for all
$\boldsymbol{\beta}'\leq\boldsymbol{\beta}$.

Let $(L,g)$ be a CS/AC manifold. Then, setting
$\boldsymbol{\beta}=(\boldsymbol{\mu},\boldsymbol{\lambda})$, the same
conclusions hold for $\boldsymbol{\mu}'\leq\boldsymbol{\mu}$ on the CS ends and
$\boldsymbol{\lambda}'\geq\boldsymbol{\lambda}$ on the AC ends.
\end{theorem}


\subsection{Cohomology of manifolds with ends}\label{ss:closedforms}

Any smooth compact manifold or smooth manifold with ends $L$ has topology of
finite type. In particular, the first cohomology group
$$H^1(L;\R):=\frac{\{\mbox{Smooth closed $1$-forms on $L$}\}}{d(C^\infty(L))}$$ 
has finite dimension $b^1(L)$, proving the following statement concerning the
structure of the space of smooth closed $1$-forms.

\begin{decomp}[for compact manifolds or manifolds with
ends]\label{decomp:closedforms} 
Let $L$ be a smooth compact manifold or a smooth manifold with ends. Choose a
finite-dimensional vector space $H$ of closed $1$-forms on $L$ such that the map
\begin{equation}
H\rightarrow H^1(L;\R), \ \ \alpha\mapsto [\alpha]
\end{equation}
is an isomorphism. Then
\begin{equation}\label{eq:closedforms}
\{\mbox{Smooth closed $1$-forms on $L$}\}= H\oplus d(C^\infty(L)). 
\end{equation}
\end{decomp}

We now want to show that in the case of a manifold with ends there exist natural
conditions on the space of 1-forms $H$. 

\begin{definition}\label{def:translationinvariant}
Given a manifold $\Sigma$, set $C:=\Sigma\times (0,\infty)$. Consider the
projection $\pi:\Sigma\times (0,\infty)\rightarrow \Sigma$. A $p$-form $\eta$ on
$C$ is \textit{translation-invariant} if it is of the form $\eta=\pi^*\eta'$,
for some $p$-form $\eta'$ on $\Sigma$.
\end{definition}

\begin{lemma}\label{lemma:formclosed}
Let $L$ be a smooth manifold with ends $S_i$. Let $\alpha$ be a smooth closed
1-form on $L$. Then there exist a smooth closed 1-form $\alpha'$ and a smooth
function $A$ on $L$ such that $\alpha'_{|S_i}$ is translation-invariant and
$\alpha=\alpha'+dA$. 
If furthermore $\alpha$ has compact support then we can choose $\alpha'$ to have
compact support.
\end{lemma}
\begin{proof}
The proof follows the scheme of the Poincar\'e Lemma for de Rham cohomology, cf.
\textit{e.g.} \cite{botttu}. Given any $p$-form $\eta$ on $S_i=\Sigma_i\times
(1,\infty)$, we can write 
$$\eta=\eta_1(\theta,r)+\eta_2(\theta,r)\wedge dr$$
for some $r$-dependent $p$-form $\eta_1$ and ($p-1$)-form $\eta_2$ on $\Sigma$.
Specifically, $\eta_1$ is the restriction of $\eta$ to the cross-sections
$\Sigma_i\times\{r\}$ and $\eta_2:=i_{\partial r} \eta$. For a fixed $R_0>1$ we
then define $(K\eta)(\theta,r):=\int_{R_0}^r\eta_2(\theta,\rho)\,d\rho$. 

Let us apply this to the 1-form obtained by restricting $\alpha$ to $S_i$,
writing
$$\alpha_{|S_i}=\alpha_1(\theta,r)+\alpha_2(\theta,r)\, dr$$
for some $r$-dependent 1-form $\alpha_1$ and function $\alpha_2$ on $\Sigma_i$.
It is then easy to check that
\begin{eqnarray*}
d\alpha_{|S_i} &=& d_\Sigma\alpha_1-(\frac{\partial}{\partial r} \alpha_1)\wedge
dr+(d_\Sigma\alpha_2)\wedge dr,\\
K\alpha_{|S_i} &=&\int_{R_0}^r\alpha_2(\theta,\rho)\,d\rho,\\
d(K\alpha_{|S_i})
&=&\int_{R_0}^rd_\Sigma\alpha_2(\theta,\rho)\,d\rho+\alpha_2(\theta,r)\,dr.
\end{eqnarray*}
From $d\alpha=0$ it follows that $\alpha_1(\theta,R_0)+d(K\alpha)=\alpha_{|S_i}$
and that $\alpha_1(\theta,R_0)$ is closed. Setting
$\alpha'_i:=\alpha_1(\theta,R_0)$ and $A_i:=K\alpha$ we can rewrite this as
$\alpha_{|S_i}=\alpha'_i+dA_i$. Interpolating between the $A_i$ yields a global
smooth function $A$ on $L$ such that $\alpha_{|S_i}=\alpha'_i+dA_{|S_i}$. We can
now define $\alpha':=\alpha-dA$ to obtain the global relationship
$$\alpha=\alpha'+dA.$$
It is clear from this construction that if $\alpha$ has compact support then
(choosing $R_0$ large enough) $\alpha'$ also has compact support.
\end{proof}

Recall that compactly-supported forms give rise to the following theory. Let $L$
be a smooth manifold with ends. We denote by $\Lambda^p_c(L;\R)$ the space of
smooth compactly-supported $p$-forms on $L$ and by $H^p_c(L;\R)$ the
corresponding cohomology groups. Let $\Sigma$ denote the union of the links of
$L$. Notice that $L$ is deformation-equivalent to a compact manifold with
boundary $\Sigma$. Standard algebraic topology (see also \cite{joyce:I} Section
2.4) proves that the inclusion $\Sigma\subset L$ gives rise to a long exact
sequence in cohomology
\begin{equation}\label{eq:cohomsequence}
0\rightarrow H^0(L;\R)\rightarrow H^0(\Sigma;\R)\stackrel{\delta}{\rightarrow}
H^1_c(L;\R)\stackrel{\gamma}{\rightarrow} H^1(L;\R)\stackrel{\rho}{\rightarrow}
H^1(\Sigma;\R)\rightarrow\dots.
\end{equation}
Here, $\gamma$ is induced by the injection $\Lambda^1_c(L;\R) \rightarrow
\Lambda^1(L;\R)$ and $\rho$ is induced by the restriction
$\Lambda^1(L;\R)\rightarrow \Lambda^1(\Sigma;\R)$. We set
$\widetilde{H}^1_c:=\mbox{Im}(\gamma)=\mbox{Ker}(\rho)$. Exactness implies that 
\begin{align}\label{eq:dimexactseq}
\mbox{dim}(\widetilde{H}^1_c) &=
\mbox{dim}(H^1_c(L;\R))-\mbox{dim}(H^0(\Sigma;\R))+\mbox{dim}(H^0(L;\R))\\
\nonumber
&= b^1_c(L)-e+1.
\end{align}
\begin{remark} \label{rem:compactsupport}
The sequence \ref{eq:cohomsequence} shows that 
\begin{equation}\label{eq:exactseq}
H^1_c(L,\R)\simeq \widetilde{H}^1_c\oplus \mbox{Ker}(\gamma) = \widetilde{H}^1_c
\oplus \mbox{Im}(\delta).
\end{equation}
This decomposition can be expressed in words as follows. By definition,
$H^1_c(L;\R)$ is determined by the classes of compactly-supported 1-forms which
are not the differential of a compactly-supported function. Given any such form,
there are two cases: (i) it is not the differential of \textit{any} function, in
which case $\gamma$ maps its class to a non-zero element of $\widetilde{H}^1_c$,
(ii) it is the differential of some function, in which case $\gamma$ maps its
class to zero. However, this function is necessarily constant on the ends of
$L$: these constants can be parametrized via $H^0(\Sigma;\R)$. Notice that the
function is only well-defined up to a constant; likewise, $\mbox{Im}(\delta)$
coincides with $H^0(\Sigma;\R)$ only up to $H^0(L;\R)\simeq\R$.
\end{remark}

Concerning Decomposition \ref{decomp:closedforms}, we can now choose $H$ as follows. For $i=1,\dots,k=\mbox{dim}(\tilde{H}^1_c)$
let $[\alpha_i]$ be a basis of $\widetilde{H}^1_c$. According to Lemma
\ref{lemma:formclosed} we can choose $\alpha_i'$ with compact support such that
[$\alpha_i']=[\alpha_i]$. For $i=1,\dots,N=\mbox{dim}(H^1)$ let $[\alpha_i]$
denote an extension to a basis of $H^1(L;\R)$. Again using Lemma
\ref{lemma:formclosed} we can choose an extension $\alpha_i'$ of
translation-invariant 1-forms such that [$\alpha_i']=[\alpha_i]$. Set
\begin{equation}\label{eq:naturalH}
\widetilde{H}:=\mbox{span}\{\alpha_1',\dots,\alpha_k'\},\ \
H:=\mbox{span}\{\alpha_1',\dots,\alpha_N'\}.
\end{equation}
Then $H$ satisfies the assumptions of Decomposition \ref{decomp:closedforms}.
One advantage of this choice of $H$ is that it reflects the relationship of
$\widetilde{H}^1_c$ to $H^1$. Specifically, if we apply Decomposition
\ref{decomp:closedforms} to $\alpha$ writing $\alpha=\alpha'+dA$ with
$\alpha'\in H$, then $[\alpha]\in \widetilde{H}^1_c$ iff $\alpha'\in
\widetilde{H}$, \textit{i.e.} iff $\alpha'$ has compact support.


\subsection{Cohomology of conifolds}\label{ss:closedforms_bis}

We now want to achieve analogous decompositions for CS and AC manifolds, in
terms of weighted spaces of closed and exact $1$-forms. 

\begin{lemma}\label{l:translationinvariant}
Let $(\Sigma,g')$ be a Riemannian manifold. Let the corresponding cone $C$ have
the conical metric $\tg:=dr^2+r^2g'$. Then any translation-invariant $p$-form
$\eta=\pi^*\eta'$ belongs to the weighted space $C^\infty_{(-p,-p)}(\Lambda^p)$.
For any $\beta>0$, $\eta$ belongs to the smaller weighted space
$C^\infty_{(-p+\beta,-p-\beta)}(\Lambda^p)$ iff $\eta'=0$.
\end{lemma}
\begin{proof} As seen in the proof of Lemma \ref{lemma:formclosed}, the general
$p$-form $\eta$ on $C$ can be written
$\eta=\eta_1(\theta,r)+\eta_2(\theta,r)\wedge dr$. The form is
translation-invariant iff $\eta_1$ is $r$-independent and $\eta_2=0$. In this
case $|\eta|_{\tg}=r^{-p}|\eta_1|_{g'}$ so $|\eta|_{\tg}=O(r^{-p})$ both for
$r\rightarrow 0$ and for $r\rightarrow \infty$. This proves that $\eta\in
C^0_{(-p,-p)}(\Lambda^p)$. To show that $\eta\in C^\infty_{(-p,-p)}(\Lambda^p)$
it is necessary to estimate $|\tnabla^k\eta|_{\tg}$, where $\tnabla$ is the
Levi-Civita connection. This can be done fairly explicitly in terms of
Christoffel symbols. In particular one can choose local coordinates on
$U\subset\Sigma$ defining a local frame $\partial_1,\cdots,\partial_{m-1}$. Set
$\partial_0:=\partial r$, the standard frame on $(0,\infty)$. The Christoffel
symbols for the corresponding frame on $(0,\infty)\times U$ and the metric $\tg$
can then be computed explicitly: for $i,j,k\geq 1$ one finds that
$\tGamma_{i,j}^k$ is bounded, $\tGamma_{i,j}^0=O(r)$,
$\tGamma_{i,0}^k=O(r^{-1})$,
$\tGamma_{0,0}^k=\tGamma_{i,0}^0=\tGamma_{0,0}^0=0$. The Christoffel symbols
defined by $\tg$ for the other tensor bundles depend linearly on these, so they
have the same bounds. Using these calculations one finds that
$|\tnabla^k\eta|_{\tg}=O(r^{-p-k})$, as desired.

It is clear from the proof that $\eta$ satisfies stronger bounds iff it
vanishes.
\end{proof}

\begin{decomp}[for CS or AC manifolds and forms with allowable
growth]\label{decomp:closedforms_growth}
Let $L$ be a CS manifold. Choose a finite-dimensional vector space $H$ of smooth
closed 1-forms on $L$ as in Equation \ref{eq:naturalH}. Then, for any
$\boldsymbol{\beta}<0$,
\begin{equation}\label{eq:closedforms_growth_cs}
\{\mbox{Closed 1-forms on $L$ in $C^\infty_{\boldsymbol{\beta}-1}(\Lambda^1)$}\}
= H\oplus d(C^\infty_{\boldsymbol{\beta}}(L)).
\end{equation}
Analogously, let $L$ be an AC manifold. Choose $H$ as above. Then, for any
$\boldsymbol{\beta}>0$,
\begin{equation}\label{eq:closedforms_growth_ac}
\{\mbox{Closed 1-forms on $L$ in $C^\infty_{\boldsymbol{\beta}-1}(\Lambda^1)$}\}
= H\oplus d(C^\infty_{\boldsymbol{\beta}}(L)).
\end{equation}
\end{decomp}

\begin{proof} 
Consider the CS case. Since $\boldsymbol{\beta}<0$, Lemma
\ref{l:translationinvariant} proves that $H\oplus
d(C^\infty_{\boldsymbol{\beta}}(L))\subseteq \{\mbox{Closed 1-forms in
$C^\infty_{\boldsymbol{\beta}-1}(\Lambda^1)$}\}$. Now choose a closed $\alpha\in
C^\infty_{\boldsymbol{\beta}-1}(\Lambda^1)$. By Decomposition
\ref{decomp:closedforms} we can write $\alpha=\alpha'+dA$, for some $\alpha'\in
H$ and $A\in C^\infty(L)$. Notice that $dA=\alpha-\alpha'\in
C^\infty_{\boldsymbol{\beta}-1}(\Lambda^1)$. By integration, again using the
fact $\boldsymbol{\beta}<0$, we conclude that $A\in
C^\infty_{\boldsymbol{\beta}}(L)$. This proves the opposite inclusion, thus the
identity. The AC case is analogous.
\end{proof}
 
\begin{lemma}\label{l:formexact}
Assume $L$ is a CS manifold. If $\alpha$ is a smooth closed 1-form on $L$
belonging to the space $C^\infty_{\boldsymbol{\beta}-1}(\Lambda^1)$ for some
$\boldsymbol{\beta}>0$ then there exists a smooth closed 1-form $\alpha'$ with
compact support on $L$ and a smooth function $A\in
C^\infty_{\boldsymbol{\beta}}(L)$ such that $\alpha=\alpha'+dA$. 

Assume $L$ is an AC manifold. If $\alpha$ is a smooth closed 1-form on $L$
belonging to the space $C^\infty_{\boldsymbol{\beta}-1}(\Lambda^1)$ for some
$\boldsymbol{\beta}<0$ then there exists a smooth closed 1-form $\alpha'$ with
compact support on $L$ and a smooth function $A\in
C^\infty_{\boldsymbol{\beta}}(L)$ such that $\alpha=\alpha'+dA$. 
\end{lemma}
\begin{proof} The proof is a variation of the proof of Lemma
\ref{lemma:formclosed}, as follows. Consider the AC case. Write
$\alpha_{|S_i}=\alpha_1+\alpha_2\wedge dr$. Define
$K\alpha:=-\int_r^\infty\alpha_2(\theta,\rho)\,d\rho$: this converges because
$\boldsymbol{\beta}<0$. It is simple to check that $d(K\alpha)=\alpha$; in
particular, this shows that $\alpha$ is exact on each end $S_i$. Setting
$A:=K\alpha$ and extending as in Lemma \ref{lemma:formclosed} leads to a global
decomposition $\alpha=\alpha'+dA$ on $L$. By construction $\alpha'$ has compact
support and $A\in C^\infty_{\boldsymbol{\beta}}$. The CS case is analogous, with
$K\alpha:=\int_0^r\alpha_2(\theta,\rho)\,d\rho$.
\end{proof}

\begin{decomp}[for CS or AC manifolds and forms with allowable
decay]\label{decomp:closedforms_decay}
Let $L$ be a CS manifold. Assume $\boldsymbol{\beta}>0$. Choose a
finite-dimensional vector space $H$ of closed 1-forms on $L$ as in Equation
\ref{eq:naturalH}, using $\widetilde{H}_0$ to denote the space $\widetilde{H}$.
For any $i=1,\dots,e$ choose a smooth function $f_i$ on $L$ such that $f_i\equiv
1$ on the end $S_i$ and $f_i\equiv 0$ on the other ends. We can do this in such
a way that $\sum f_i\equiv 1$. Let $E_0$ denote the $e$-dimensional vector space
generated by these functions. By construction $E_0$ contains the constant
functions so $d(E_0)$ has dimension $e-1$. It is simple to check that
$d(E_0)\cap d(C^\infty_{\boldsymbol{\beta}}(L))=\{0\}$. Then
\begin{equation}\label{eq:closedforms_decay_cs}
\{\mbox{Closed 1-forms on $L$ in
$C^\infty_{\boldsymbol{\beta}-1}(\Lambda^1)$}\}=\widetilde{H}_0\oplus
d(E_0)\oplus d(C^\infty_{\boldsymbol{\beta}}(L)).
\end{equation}
Analogously, let $L$ be an AC manifold. Assume $\boldsymbol{\beta}<0$. Choose
spaces as above, this time using the notation $\widetilde{H}_\infty$ and
$E_\infty$. Then
\begin{equation}\label{eq:closedforms_decay_ac}
\{\mbox{Closed 1-forms on $L$ in $C^\infty_{\boldsymbol{\beta}-1}(\Lambda^1)$}\}
= \widetilde{H}_\infty\oplus d(E_\infty)\oplus
d(C^\infty_{\boldsymbol{\beta}}(L)).
\end{equation}
\end{decomp}
\begin{proof} Consider the CS case. The inclusion $\supseteq$ is clear.
Conversely, let $\alpha\in C^\infty_{\boldsymbol{\beta}-1}(\Lambda^1)$ be
closed. Decomposition \ref{decomp:closedforms} allows us to write
$\alpha=\alpha'+dA$, for some uniquely defined $\alpha'\in H$ and some $A\in
C^\infty(L)$, well-defined up to a constant. Lemma \ref{l:formexact} implies
that the cohomology class of $\alpha$ belongs to the space $\widetilde{H}^1_c$,
\textit{i.e.} that $\alpha'\in \widetilde{H}_0$ so it has compact support. This
shows that $dA\in C^\infty_{\boldsymbol{\beta}-1}(\Lambda^1)$. Writing
$A_i:=A_{|S_i}$ we find $dA_i=d_{\Sigma_i}A_i+\frac{\partial A_i}{\partial
r}\,dr$, thus $\frac{\partial A_i}{\partial r}\in
C^\infty_{\boldsymbol{\beta}-1}(L)$. This shows that $\int_0^r\frac{\partial
A_i}{\partial r}\,d\rho\in C^\infty_{\boldsymbol{\beta}}(L)$. This determines
$A_i$ up to a constant $c_i$ on each end. Together with Equation
\ref{eq:exactseq} this proves the claim. The AC case is analogous.
\end{proof}

We now turn to the case of CS/AC manifolds, concentrating on the situations of
most interest to us.

\begin{decomp}[for CS/AC manifolds]\label{decomp:closedforms_accs}
Let $L$ be a CS/AC manifold with $s$ CS ends and $l$ AC ends. As usual we denote
the union of the CS links by $\Sigma_0$ and the union of the AC links by
$\Sigma_\infty$. Choose a finite-dimensional vector space $H$ of closed 1-forms
on $L$ as in Equation \ref{eq:naturalH}, using $\widetilde{H}_{0,\infty}$ to
denote the space $\widetilde{H}$. For any $i=1,\dots,s+l$ choose a function
$f_i$ such that $f_i\equiv 1$ on the end $S_i$ and $f_i\equiv 0$ on the other
ends. We can assume that $\sum f_i\equiv 1$. Let $E_{0,\infty}$ denote the
$(s+l)$-dimensional vector space generated by these functions. Then, for any
$\boldsymbol{\mu}>0$ and $\boldsymbol{\lambda}<0$,
\begin{equation}\label{eq:closedforms_decay_accs}
\{\mbox{Closed 1-forms on $L$ in
$C^\infty_{(\boldsymbol{\mu}-1,\boldsymbol{\lambda}-1)}(\Lambda^1)$}\}=
\widetilde{H}_{0,\infty}\oplus d(E_{0,\infty})\oplus
d(C^\infty_{(\boldsymbol{\mu},\boldsymbol{\lambda})}(L)). 
\end{equation}
Now let $\Lambda^p_{c,\bullet}(L;\R)$ denote the space of $p$-forms on $L$ which
vanish in a neighbourhood of the singularities, with no condition on the large
ends. Let $H^p_{c,\bullet}(L;\R)$ denote the corresponding cohomology groups.
Let $\widetilde{H}^1_{c,\bullet}$ denote the image of the map
$\gamma:H^1_{c,\bullet}(L;\R)\rightarrow H^1(L;\R)$. Choose a finite-dimensional
vector space $\widetilde{H}_{0,\bullet}$ of translation-invariant closed 1-forms
on $L$ with compact support in a neighbourhood of the singularities and such
that the map
\begin{equation}
\widetilde{H}_{0,\bullet}\rightarrow \widetilde{H}^1_{c,\bullet}, \ \
\alpha\mapsto [\alpha]
\end{equation}
is an isomorphism. For any $i=1,\dots,s$ choose a function $f_i$ such that
$f_i\equiv 1$ on the CS end corresponding to the singularity $x_i$ and
$f_i\equiv 0$ on the other ends. Let $E_0$ denote the s-dimensional vector space
generated by these functions. Then, for any $\boldsymbol{\mu}>0$ and
$\boldsymbol{\lambda}>0$,
\begin{equation}\label{eq:closedforms_growth_accs}
\{\mbox{Closed 1-forms on $L$ in
$C^\infty_{(\boldsymbol{\mu}-1,\boldsymbol{\lambda}-1)}(\Lambda^1)$}\}=
\widetilde{H}_{0,\bullet}\oplus d\Big(E_0\oplus
C^\infty_{(\boldsymbol{\mu},\boldsymbol{\lambda})}(L)\Big). 
\end{equation}
\end{decomp}
\proof{} The proof is similar to the proofs of the previous decompositions. It
may however be good to emphasize that, in the case $\boldsymbol{\mu}>0$ and
$\boldsymbol{\lambda}>0$, $d(E_0)\cap
d(C^\infty_{(\boldsymbol{\mu},\boldsymbol{\lambda})}(L))\neq\{0\}$ (it is
one-dimensional). This explains the slightly different statement of
Decomposition \ref{eq:closedforms_growth_accs}.
\endproof

\begin{remark}
The weight $\boldsymbol{\beta}=0$ corresponds to an exceptional case in Lemma
\ref{l:formexact}: integration will generally generate log terms, so we cannot
conclude that $A\in C^\infty_{\boldsymbol{\beta}}$ there. One can analogously
argue that
$C^\infty_{-\boldsymbol{1}}(\Lambda^1)/d(C^\infty_{\boldsymbol{0}}(L))$ is not
finite-dimensional. 

Similar decompositions hold for $k$-forms: in this setting the exceptional case
corresponds to $\boldsymbol{\beta}=k-1$.
\end{remark} 

\begin{remark}
Notice that the above decompositions do not cover all possibilities: for
example, given a CS manifold we could decide to study the space of closed
1-forms in $C^\infty_{\boldsymbol{\beta}-1}(\Lambda^1)$ corresponding to a
weight $\boldsymbol{\beta}=(\beta_1,\dots,\beta_e)$ with some $\beta_i$ positive
and others negative. However, it should be clear from the above discussion how
to use the same ideas to cover any other case of interest. We have restricted
our attention to the cases most relevant to this paper.
\end{remark}

For future reference it is useful to emphasize the topological interpretation of
some of the previous results. The reasons underlying our interest for each case
will become apparent in Section \ref{s:moduli}.

\begin{corollary}\label{cor:topsummary}
Let $L$ be a smooth compact manifold. Then 
\begin{equation*}
\{\mbox{Closed $1$-forms on $L$}\}\simeq H^1(L;\R)\oplus d(C^\infty(L)). 
\end{equation*}
Let $(L,g)$ be an AC manifold. Then for $\boldsymbol{\beta}<0$
\begin{equation*}
\{\mbox{Closed $1$-forms on $L$ in
$C^\infty_{\boldsymbol{\beta}-1}(\Lambda^1)$}\} \simeq H^1_c(L;\R) \oplus
d(C^\infty_{\boldsymbol{\beta}}(L)), 
\end{equation*}
while for $\boldsymbol{\beta}>0$
\begin{equation*}
\{\mbox{Closed $1$-forms on $L$ in
$C^\infty_{\boldsymbol{\beta}-1}(\Lambda^1)$}\} \simeq H^1(L;\R) \oplus
d(C^\infty_{\boldsymbol{\beta}}(L)).
\end{equation*}
Let $(L,g)$ be a CS manifold with link $\Sigma_0$. Then for
$\boldsymbol{\beta}>0$
\begin{align*}
&\{\mbox{Closed $1$-forms on $L$ in
$C^\infty_{\boldsymbol{\beta}-1}(\Lambda^1)$}\}\\
&\quad\simeq \mbox{Ker}\left(H^1(L)\stackrel{\rho}{\rightarrow}
H^1(\Sigma_0)\right) \oplus d(E_0)\oplus d(C^\infty_{\boldsymbol{\beta}}(L)).
\end{align*}
Let $(L,g)$ be a CS/AC manifold with link $\Sigma=\Sigma_0\amalg\Sigma_\infty$.
Then for $\boldsymbol{\mu}>0$ and $\boldsymbol{\lambda}<0$
\begin{align*}
&\{\mbox{Closed $1$-forms on $L$ in
$C^\infty_{(\boldsymbol{\mu}-1,\boldsymbol{\lambda}-1)}(\Lambda^1)$}\}\\
&\quad\simeq \mbox{Ker}\left(H^1_{\bullet,c}(L)\stackrel{\rho}{\rightarrow}
H^1(\Sigma_0)\right)\oplus d(E_0) \oplus
d(C^\infty_{(\boldsymbol{\mu},\boldsymbol{\lambda})}(L)),
\end{align*}
while for $\boldsymbol{\mu}>0$ and $\boldsymbol{\lambda}>0$
\begin{align*}
&\{\mbox{Closed $1$-forms on $L$ in
$C^\infty_{(\boldsymbol{\mu}-1,\boldsymbol{\lambda}-1)}(\Lambda^1)$}\}\\
&\quad\simeq \mbox{Ker}\left(H^1(L)\stackrel{\rho}{\rightarrow}
H^1(\Sigma_0)\right)\oplus d\left(E_0\oplus
C^\infty_{(\boldsymbol{\mu},\boldsymbol{\lambda})}(L)\right).
\end{align*}
\end{corollary}
\begin{proof}
The compact case coincides with Equation \ref{eq:closedforms}. The AC case with
$\boldsymbol{\beta}<0$ follows from Equation \ref{eq:closedforms_decay_ac} and
Remark \ref{rem:compactsupport}. The AC case with $\boldsymbol{\beta}>0$
coincides with Equation \ref{eq:closedforms_growth_ac}. The CS case coincides
with Equation \ref{eq:closedforms_decay_cs}. 

Let us now focus on the CS/AC case with $\boldsymbol{\lambda}<0$. Using the
notation of Decomposition \ref{decomp:closedforms_accs}, let $E'$ denote a
complement of $E_0\oplus\R$ in $E_{0,\infty}$, \textit{i.e.}
$E_{0,\infty}=E_0\oplus\R\oplus E'$. Notice that the long exact sequence
\ref{eq:cohomsequence} with $\Sigma=\Sigma_0\amalg\Sigma_\infty$ leads to an
identification $H^1_c(L;\R)\simeq \widetilde{H}^1_c(L)\oplus d(E_{0,\infty})$.
One can also set up the ``relative" analogue of Sequence \ref{eq:cohomsequence}
using the inclusion of pairs $(\Sigma_0,\emptyset)\subset (L,\Sigma_\infty)$.
Using notation analogous to that of Decomposition \ref{decomp:closedforms_accs}
this leads to the long exact sequence
\begin{equation*}
0\rightarrow H^0_c(L;\R)\rightarrow H^0_{\bullet,c}(L;\R)\rightarrow
H^0(\Sigma_0;\R)\rightarrow H^1_c(L;\R)\stackrel{\gamma}{\rightarrow}
H^1_{\bullet,c}(L;\R)\stackrel{\rho}{\rightarrow}
H^1(\Sigma_0;\R)\rightarrow\dots
\end{equation*}
Since $H^0_c(L;\R)=0$ and $H^0_{\bullet,c}(L;\R)=0$, one obtains an
identification $H^1_c(L;\R)\simeq
E_0\oplus\mbox{Ker}\left(H^1_{\bullet,c}(L)\stackrel{\rho}{\rightarrow}
H^1(\Sigma_0)\right)$. Comparing these identifications yields an identification
$\widetilde{H}^1_c(L;\R)\oplus
d(E')\simeq\mbox{Ker}\left(H^1_{\bullet,c}(L)\stackrel{\rho}{\rightarrow}
H^1(\Sigma_0)\right)$. The claim follows. 

Now consider the CS/AC case with $\boldsymbol{\lambda}>0$. The long exact
sequence \ref{eq:cohomsequence} with $\Sigma=\Sigma_0$ yields
\begin{equation}\label{eq:cohomsequence_accsbis}
0\rightarrow H^0(L;\R)\rightarrow H^0(\Sigma_0;\R)\rightarrow
H^1_{c,\bullet}(L;\R)\stackrel{\gamma}{\rightarrow}
H^1(L;\R)\stackrel{\rho}{\rightarrow} H^1(\Sigma_0;\R)\rightarrow\dots
\end{equation}
This proves the final claim.
\end{proof}

\begin{remark}\label{rem:topsummary}
 Compare Equations \ref{eq:closedforms_decay_cs}, \ref{eq:closedforms_decay_ac} with the corresponding equations in the statement of Corollary \ref{cor:topsummary}. When working with AC manifolds we choose to group the two topological terms of Equation \ref{eq:closedforms_decay_ac} into one space $H^1_c(L;\R)$. When working with CS manifolds we prefer to keep the two topological terms of Equation \ref{eq:closedforms_decay_cs} separate and to emphasize the ``geometric'' meaning of one of them as kernel of a certain restriction map. These choices are based on the different roles that these spaces will play in Section \ref{s:moduli}, cf. also Remark \ref{rem:coh_dim}. 
\end{remark}


\section{Lagrangian conifolds}\label{s:lagconifolds}
A priori, a \textit{CS/AC submanifold} might simply be defined as an immersed
submanifold whose topology and induced metric is of the type defined in Section
\ref{ss:accs_review}. However, for the purposes of this article it is convenient
to strengthen the hypotheses by adding the requirement that the submanifold have
a well-defined cone at each singularity and at each end. The precise definitions
are as follows. We restrict our attention to Lagrangian submanifolds in K\"ahler
ambient spaces, but it is clear how one might extend these definitions to other
settings.
\begin{definition} \label{def:lagr}
Let $(M^{2m},\omega)$ be a symplectic manifold. An embedded or immersed
submanifold $\iota:L^m\rightarrow M$ is \textit{Lagrangian} if
$\iota^*\omega\equiv 0$. The immersion allows us to view the tangent bundle $TL$
of $L$ as a subbundle of $TM$ (more precisely, of $\iota^*TM$). When $M$ is
K\"ahler with structures $(g,J,\omega)$ it is simple to check that $L$ is
Lagrangian iff $J$ maps $TL$ to the normal bundle $NL$ of $L$, \textit{i.e.}
$J(TL)=NL$. 
\end{definition}

\begin{definition}\label{def:aclagsub}
Let $L^m$ be a smooth manifold. Assume given a Lagrangian immersion
$\iota:L\rightarrow \C^m$, the latter endowed with its standard structures
$\tilde{J},\tilde{\omega}$. We say that $L$ is an \textit{asymptotically conical Lagrangian submanifold}
with \textit{rate} $\boldsymbol{\lambda}$ if it satisfies the following
conditions.
\begin{enumerate}
\item We are given a compact subset $K\subset L$ such that $S:=L\setminus K$ has
a finite number of connected components $S_1,\dots,S_e$.
\item We are given Lagrangian cones $\mathcal{C}_i\subset \C^m$ with smooth
connected links $\Sigma_i:=\mathcal{C}_i\bigcap \Sph^{2m-1}$. Let
$\iota_i:\Sigma_i\times (0,\infty)\rightarrow \C^m$ denote the natural
immersions, parametrizing $\mathcal{C}_i$.
\item We are finally given an $e$-tuple of \textit{convergence rates}
$\boldsymbol{\lambda}=(\lambda_1,\dots,\lambda_e)$ with $\lambda_i<2$, \textit{centers} $p_i\in\C^m$ and
diffeomorphisms $\phi_i:\Sigma_i\times [R,\infty)\rightarrow \overline{S_i}$
for some $R>0$ such that, for $r\rightarrow\infty$ and all $k\geq 0$,
\begin{equation}\label{eq:aclagdecay}
|\tnabla^k(\iota\circ\phi_i-(\iota_i+p_i)|=O(r^{\lambda_i-1-k})
\end{equation}
with respect to the conical metric $\tg_i$ on $\cone_i$.
\end{enumerate}
\end{definition}

Notice that the restriction $\lambda_i<2$ ensures that the cone is unique but is
weak enough to allow the submanifold to converge to a translated copy
$\mathcal{C}_i+p'_i$ of the cone (\textit{e.g.} if $\lambda_i=1$), or even to
slowly pull away from the cone (if $\lambda_i>1$). 

\begin{definition}\label{def:cslagsub}
Let $\bar{L}^m$ be a smooth manifold except for a finite number of possibly singular
points $\{x_1,\dots,x_e\}$. Assume given a continuous map
$\iota:\bar{L}\rightarrow \C^m$ which restricts to a smooth Lagrangian immersion of
$L:=\bar{L}\setminus\{x_1,\dots,x_e\}$. We say that $\bar{L}$ (or $L$) is a \textit{conically singular
Lagrangian submanifold} with \textit{rate} $\boldsymbol{\mu}$ if it satisfies
the following conditions.
\begin{enumerate}
\item We are given open connected neighbourhoods $S_i$ of $x_i$.
\item We are given Lagrangian cones $\mathcal{C}_i\subset \C^m$ with smooth
connected links $\Sigma_i:=\mathcal{C}_i\bigcap \Sph^{2m-1}$. Let
$\iota_i:\Sigma_i\times (0,\infty)\rightarrow \C^m$ denote the natural
immersions, parametrizing $\mathcal{C}_i$.
\item We are finally given an $e$-tuple of \textit{convergence rates}
$\boldsymbol{\mu}=(\mu_1,\dots,\mu_e)$ with $\mu_i>2$, \textit{centers} $p_i\in\C^m$ and diffeomorphisms
$\phi_i:\Sigma_i\times (0,\epsilon]\rightarrow \overline{S_i}\setminus\{x_i\}$ such that, for
$r\rightarrow 0$ and all $k\geq 0$,
\begin{equation}\label{eq:cslagdecay}
|\tnabla^k(\iota\circ\phi_i-(\iota_i+p_i))|=O(r^{\mu_i-1-k})
\end{equation}
with respect to the conical metric $\tg_i$ on $\cone_i$. Notice that our assumptions imply that $\iota(x_i)=p_i$.
\end{enumerate}
\end{definition}

It is simple to check that AC Lagrangian submanifolds, with the induced metric,
satisfy Definition \ref{def:ac_manifold} with $\nu_i=\lambda_i-2$. The analogous
fact holds for CS Lagrangian submanifolds.

\begin{definition} \label{def:accslagsub}
Let $\bar{L}^m$ be a smooth manifold except for a finite number of possibly singular
points $\{x_1,\dots,x_s\}$ and with $l$ ends. Assume
given a continuous map $\iota:\bar{L}\rightarrow \C^m$ which restricts to a smooth Lagrangian
immersion of $L:=\bar{L}\setminus\{x_1,\dots,x_s\}$. We say that $\bar{L}$ (or
$L$) is a \textit{CS/AC Lagrangian submanifold} with \textit{rate}
$(\boldsymbol{\mu},\boldsymbol{\lambda})$ if in a neighbourhood of the points
$x_i$ it has the structure of a CS submanifold with rates $\mu_i$ and in a
neighbourhood of the remaining ends it has the structure of an AC submanifold
with rates $\lambda_i$.

We use the generic term \textit{Lagrangian conifold} (even though ``subconifold"
would be more appropriate) to indicate any CS, AC or CS/AC Lagrangian
submanifold.
\end{definition}

\begin{example} \label{e:lagcone}
Let $\mathcal{C}$ be a cone in $\C^m$ with smooth link $\Sigma^{m-1}$. It can be
shown that $\mathcal{C}$ is a Lagrangian iff $\Sigma$ is \textit{Legendrian} in
$\Sph^{2m-1}$ with respect to the natural \textit{contact structure} on the
sphere. Then $\mathcal{C}$ is a CS/AC Lagrangian submanifold of $\C^m$ with rate
$(\boldsymbol{\mu,\lambda})$ for any $\boldsymbol{\mu}$ and
$\boldsymbol{\lambda}$.
\end{example}

The definition of CS Lagrangian submanifolds can be generalized to K\"ahler
ambient spaces as follows. Once again we denote the standard structures on
$\C^m$ by $\tilde{J}$, $\tilde{\omega}$.
\begin{definition} \label{def:generalcslagsub}
Let $(M^{2m},J,\omega)$ be a K\"ahler manifold and $\bar{L}^m$ be a smooth
manifold except for a finite number of possibly singular points $\{x_1,\dots,x_e\}$. Assume given a continuous map $\iota:\bar{L}\rightarrow M$ which restricts
to a smooth Lagrangian immersion of $L:=\bar{L}\setminus\{x_1,\dots,x_e\}$. We say that $\bar{L}$ (or
$L$) is a \textit{Lagrangian submanifold with conical singularities} (CS
Lagrangian submanifold) if it satisfies the following conditions.
\begin{enumerate}
\item We are given isomorphisms $\upsilon_i:\C^m\rightarrow T_{\iota(x_i)}M$ such that
$\upsilon_i^*\omega=\tilde{\omega}$ and $\upsilon_i^*J=\tilde{J}$. 

According to Darboux' theorem, cf. \textit{e.g.} \cite{weinstein}, there then
exist an open ball $B_R$ in $\C^m$ (of small radius $R$) and diffeomorphisms
$\Upsilon_i:B_R\rightarrow M$ such that $\Upsilon(0)=\iota(x_i)$,
$d\Upsilon_i(0)=\upsilon_i$ and $\Upsilon_i^*\omega=\tilde{\omega}$.
\item We are given open neighbourhoods $S_i$ of $x_i$ in $\bar{L}$. We assume $S_i$ are small, in the sense that
the compositions
\begin{equation*}
 \Upsilon_i^{-1}\circ\iota:S_i\rightarrow B_R 
\end{equation*}
are well-defined. 

We are
also given Lagrangian cones $\mathcal{C}_i\subset\C^m$ with smooth connected
links $\Sigma_i:=\mathcal{C}_i\bigcap \Sph^{2m-1}$. Let $\iota_i:\Sigma_i\times
(0,\infty)\rightarrow \C^m$ denote the natural immersions, parametrizing
$\cone_i$. 
\item We are finally given an $e$-tuple of \textit{convergence rates}
$\boldsymbol{\mu}=(\mu_1,\dots,\mu_e)$ with $\mu_i\in (2,3)$ and diffeomorphisms 
$\phi_i:\Sigma_i\times (0,\epsilon]\rightarrow \overline{S_i}\setminus\{x_i\}$ such that, as $r\rightarrow 0$ and
for all $k\geq 0$,
\begin{equation}\label{eq:generalcslagdecay}
|\tnabla^k(\Upsilon_i^{-1}\circ\iota\circ\phi_i-\iota_i)|=O(r^{\mu_i-1-k})
\end{equation}
with respect to the conical metric $\tg_i$ on $\cone_i$.
\end{enumerate}
We call $x_i$ the \textit{singularities} of $\bar{L}$ and $\upsilon_i$ the
\textit{identifications}.
\end{definition}
One can check that, when $M=\C^m$, Definition \ref{def:generalcslagsub} coincides with Definition \ref{def:cslagsub} if we choose $\Upsilon_i(x):=x+\iota(x_i)$. Notice that the local diffeomorphisms between $M$ and $\C^m$ are prescribed only up to first order. Changing the diffeomorphism $\Upsilon_i$ (while keeping $\upsilon_i$ fixed) will perturb the map $\phi_i$ (and its derivatives) by a term of order $O(r^{2-k})$. In order to make the rate be independent of the particular diffeomorphism chosen, we need to introduce a constraint on the range of $\mu_i$ ensuring that $O(r^{2-k})<O(r^{\mu_i-1-k})$, thus $\mu_i<3$.

\begin{remark}
One could also define and study AC Lagrangian submanifolds in $M$, but this
would require a preliminary study of AC metrics on K\"ahler manifolds, going
beyond the scope of this article. We refer to \cite{pacini:defs} for some
details in this direction.
\end{remark}


\section{Deformations of Lagrangian conifolds}\label{s:lagdefs}
We now want to understand how to parametrize the \textit{Lagrangian
deformations} of a given Lagrangian conifold $L\subset M$. Since the Lagrangian
condition is invariant under reparametrization of $L$, to avoid huge amounts of
geometric redundancy it is best to work in terms of non-parametrized
submanifolds; in other words, in terms of equivalence classes of immersed
submanifolds, where two immersions are \textit{equivalent} if they differ by a
reparametrization. Then, to parametrize the possible deformations of $L$, it is
sufficient to prove a \textit{Lagrangian neighbourhood theorem}. 

\begin{remark} 
The analogous situation in the Riemannian setting is well-known. The set
$\mbox{Imm}(L,M)$ of immersions $L\rightarrow (M,g)$ can be topologized via the
\textit{$C^1$} or \textit{Whitney} topology, \textit{i.e.} in terms of the
natural topology on the first jet bundle $J^1(L,M)$. The group of
diffeomorphisms $\mbox{Diff}(L)$ acts on this space by reparametrization. Choose
an element $\iota\in \mbox{Imm}(L,M)$. Let $NL$ denote the normal bundle. Using
the \textit{tubular neighbourhood theorem} one can define a natural injection
$$\Lambda^0(NL)\rightarrow\mbox{Imm}(L,M)/\mbox{Diff}(L).$$ 
In standard situations (for example when $L$ is compact) this actually defines a
local homeomorphism between the natural topologies on these spaces.
\end{remark}

A foundation for the theory of Lagrangian neighbourhoods is provided by the
following linear-algebraic construction. Let $W$ be a finite-dimensional real
vector space. Then $W\oplus W^*$ admits a canonical symplectic structure
$\hat{\omega}$ defined as follows:
\begin{equation}\label{eq:canonicallinearcotg}
\hat{\omega}(w_1+\alpha_1,w_2+\alpha_2):=\alpha_2(w_1)-\alpha_1(w_2).
\end{equation}
It turns out that this example of symplectic vector space is actually very
general, in the following sense. Let $(V,\omega)$ be a symplectic vector space.
Let $W\subset V$ be a Lagrangian subspace. Choose a Lagrangian complement
$Z\subset V$, so that $V=W\oplus Z$. It is simple to check that the restriction
of $\omega$ to $Z$ defines an isomorphism
\begin{equation}\label{eq:canonicalisgeneral}
\omega_{|Z}:Z\rightarrow W^*,\ \ z\mapsto\omega(z,\cdot)
\end{equation}
and that, using this isomorphism, one can build an isomorphism $\gamma:(W\oplus
W^*,\hat{\omega})\simeq (V,\omega)$. Furthermore, such $\gamma$ is unique if we
impose that it coincide with the identity on $W$. Adding this condition thus
implies that $\gamma$ is uniquely defined by the choice of $Z$.

It is a well-known fact, first noticed by Souriau \cite{souriau}, that a similar
construction exists also for symplectic manifolds. The construction is based on
the following standard facts. Given any manifold $L$, the cotangent bundle
$T^*L$ admits a canonical symplectic structure $\hat{\omega}$. Specifically,
consider the \textit{tautological} 1-form on $T^*L$ defined by
$\hat{\lambda}[\alpha](v):=\alpha(\pi_*(v))$, where $\pi:T^*L\rightarrow L$ is
the natural projection. Then $\hat{\omega}:=-d\hat{\lambda}$. Notice that a
section of $T^*L$ is simply a 1-form $\alpha$ on $L$. The graph $\Gamma(\alpha)$
is Lagrangian in $T^*L$ iff $\alpha$ is closed. In particular the zero section
$L\subset T^*L$ is Lagrangian. Furthermore each fibre $\pi^{-1}(p)=T^*_pL$ is a
Lagrangian submanifold. The fibres thus define a Lagrangian foliation of $T^*L$
transverse to the zero section. Finally, every 1-form $\alpha$ defines a
\textit{translation} map
\begin{equation}\label{eq:translation}
\tau_\alpha:T^*L\rightarrow T^*L,\ \ \tau_\alpha(x,\eta):=(x,\alpha(x)+\eta).
\end{equation}
If $\alpha$ is closed then this map is a symplectomorphism of
$(T^*L,\hat{\omega})$.


\subsection{First case: smooth compact Lagrangian
submanifolds}\label{ss:cptlagdefs}
We can now quote Souriau's result, following Weinstein \cite{weinstein}
Corollary 6.2.

\begin{theorem}\label{th:nbd_weinstein}
Let $(M,\omega)$ be a symplectic manifold. Let $L\subset M$ be a smooth compact
Lagrangian submanifold. Then there exist a neighbourhood $\mathcal{U}$ of the
zero section of $L$ inside its cotangent bundle $T^*L$ and an embedding
$\Phi_L:\mathcal{U}\rightarrow M$ such that $\Phi_{L|L}=Id:L\rightarrow L$ and
$\Phi_L^*\omega=\hat{\omega}$. 
\end{theorem}
\begin{proof}
For each $x\in L$, $T_xL$ is a Lagrangian subspace of $T_xM$. The first step is
to choose a Lagrangian complement $Z_x$, so that $T_xM=T_xL\oplus Z_x$. This can
be done smoothly with respect to $x$ using the fact that the space of Lagrangian
complements is a contractible set inside the Grassmannian of $m$-planes in
$T_xM$. As seen following Equation \ref{eq:canonicalisgeneral}, $\omega$ then
provides an isomorphism $\gamma_x:(T_xL\oplus T_x^*L,\hat{\omega})\rightarrow
(T_xM,\omega)$, uniquely defined by the condition that $\gamma_x=Id$ on $T_xL$.
Now choose a diffeomorphism $\Psi_L:\mathcal{U}\rightarrow M$ such that
$(\Psi_L)_*$ extends $\gamma$. By construction, the pull-back form
$(\Psi_L)^*\omega$ coincides with $\hat{\omega}$ at each point of $L$. We now
need to perturb $\Psi_L$ so that the pull-back form coincides with
$\hat{\omega}$ in a neighbourhood of $L$. Set $\omega_0:=\hat{\omega}$ and
$\omega_1:=(\Psi_L)^*\omega$. One can use an argument due to Moser together with
the Poincar\'e Lemma to prove that there exists a diffeomorphism
$k:T^*L\rightarrow T^*L$ such that $k^*\omega_1=\omega_0$ and
$k_{|L}=Id:L\rightarrow L$. Thus $\Phi_L:=\Psi_L\circ k$ has the required
properties. For later use it is also useful to note that, using the same
argument as in \cite{weinstein} Theorem 7.1, one can further show that, at each
$x\in L$, $k_*$ preserves $T_x^*L$. A linear-algebraic argument then shows that
this implies that $k_*=Id$ at each $x\in L$. Thus $(\Phi_L)_*=(\Psi_L)_*$ at
each $x\in L$.
\end{proof}
\begin{remark} Although the statement and proof are for embedded submanifolds it
is not difficult to extend them to immersed compact Lagrangian submanifolds by
working locally. In this case $\Phi_L$ will only be a local embedding. 
\end{remark}
Let $C^\infty(\mathcal{U})$ denote the space of smooth 1-forms on $L$ whose
graph lies in $\mathcal{U}$. Theorem \ref{th:nbd_weinstein} leads immediately to
the following conclusion.
\begin{corollary} \label{cor:nbd_weinstein}
Let $(M,\omega)$ be a symplectic manifold. Let $L\subset M$ be a smooth compact
Lagrangian submanifold. Then $\Phi_L$ defines by composition an injective map 
\begin{equation}\label{eq:cptlagrcorrespondence}
\Phi_L:C^\infty(\mathcal{U})\rightarrow \mbox{Imm}(L,M)/\mbox{Diff}(L).
\end{equation}
A section $\alpha\in C^\infty(\mathcal{U})$ is closed iff the corresponding
(non-parametrized) submanifold $\Phi_L\circ\alpha$ is Lagrangian.
\end{corollary}
An important point about the map $\Phi_L$ in Equation
\ref{eq:cptlagrcorrespondence} is that any submanifold which admits a
parametrization which is $C^1$-close to some parametrization of $L$ belongs to
the image of $\Phi_L$, \textit{i.e.} corresponds to a 1-form $\alpha$. 

Let $\mbox{Lag}(L,M)$ denote the set of Lagrangian immersions from $L$ into $M$.
Using Corollary \ref{cor:nbd_weinstein} and the Fr\'echet topology on
$C^\infty(\mathcal{U})$ we can locally define a topology on
$\mbox{Lag}(L,M)/\mbox{Diff}(L)$; one can then check that on the intersection of
any two open sets these topologies coincide, so we obtain a global topology on
$\mbox{Lag}(L,M)/\mbox{Diff}(L)$. The connected component containing the given
$L\subset M$ defines the \textit{moduli space of Lagrangian deformations of
$L$}. Coupling Corollary \ref{cor:nbd_weinstein} with Decomposition
\ref{decomp:closedforms} gives a good idea of the local structure of this space.


\subsection{Second case: Lagrangian cones in $\C^m$}\label{ss:conelagdefs}
Let $\mathcal{C}$ be a Lagrangian cone in $\C^m$ with link $(\Sigma,g')$ and
conical metric $\tg$. The goal of this section is to provide an analogue of the
theory of Section \ref{ss:cptlagdefs} for this specific submanifold, giving a correspondence between closed 1-forms in $C^\infty_{(\mu-1,\lambda-1)}(\Lambda^1)$ and Lagrangian deformations of $\mathcal{C}$ with rate $(\mu,\lambda)$.

Let $\theta$ denote the generic point on $\Sigma$. We will identify
$\Sigma\times (0,\infty)$ with $\mathcal{C}$ via the immersion
\begin{equation}
\iota:\Sigma\times (0,\infty)\rightarrow \C^m, \ \ (\theta,r)\mapsto r\theta.
\end{equation}
\begin{remark}\label{rem:iota_*}
Let $\theta(t)$ be a curve in $\Sigma$ such that $\theta(0)=\theta$. Let $r(t)$
be a curve in $\R^+$ such that $r(0)=r$.
Differentiating $\iota$ at the point $(\theta,r)$ gives identifications
\begin{equation}
\begin{array}{rcl}
\iota_*:T_\theta\Sigma\oplus\R & \rightarrow &
T_{r\theta}\mathcal{C}\subset\C^m\\
(\theta'(0),r'(0)) & \mapsto &
d/dt\left(r(t)\theta(t)\right)_{|t=0}=r'(0)\theta+r\theta'(0)\in\C^m.
\end{array}
\end{equation}
This leads to the general formula $\iota_{*|\theta,r}(v,a)=a\theta+rv$.
\end{remark}
We can build an explicit (local) identification $\Psi_{\mathcal{C}}$ of
$T^*\mathcal{C}$ with $\C^m$ as follows.

Firstly, the metric $\tg$ gives an identification
\begin{equation}\label{eq:cotg=tg}
T^*\mathcal{C}\rightarrow T\mathcal{C},\ \
(\theta,r,\alpha_1+\alpha_2\,dr)\mapsto (\theta,r,r^{-2}A_1+\alpha_2\partial r),
\end{equation}
where $g'(A_1,\cdot)=\alpha_1$ and we use the notation of Section
\ref{ss:closedforms}. Notice that, according to Remark \ref{rem:iota_*}, the
corresponding vector in $\C^m$ is $\iota_*(r^{-2}A_1+\alpha_2\partial
r)=\alpha_2\theta+r^{-1}A_1$. Notice also that Equation \ref{eq:cotg=tg} defines
a fibrewise isometry between vector bundles over $\mathcal{C}$. Let
$\widetilde{\nabla}$ denote the standard connection on the tangent bundle of
$\C^m$. Since $\mathcal{C}$ has the induced metric, the Levi-Civita connection
on $T\mathcal{C}$ coincides with the tangential projection
$\widetilde{\nabla}^T$. Let $T^*\mathcal{C}$ have the induced Levi-Civita
connection. Then Equation \ref{eq:cotg=tg} also defines an isomorphism between
the two connections. 

Secondly, since $\mathcal{C}$ is Lagrangian the complex structure provides an
identification
\begin{equation}\label{eq:tg=normal}
\tilde{J}:T\mathcal{C}\simeq N\mathcal{C}. 
\end{equation}
This is again a fibrewise isometry. The perpendicular component
$\widetilde{\nabla}^\perp$ defines a connection on $N\mathcal{C}$. Since $\C^m$
is K\"ahler, $\widetilde{\nabla} \tilde{J}=\tilde{J}\widetilde{\nabla}$. Thus
$\widetilde{\nabla}^\perp \tilde{J}=\tilde{J}\widetilde{\nabla}^T$, so Equation
\ref{eq:tg=normal} defines an isomorphism between the two connections.

Thirdly, the Riemannian tubular neighbourhood theorem gives an explicit (local)
identification
\begin{equation}\label{eq:normal=ambient}
N\mathcal{C}\rightarrow \C^m,\ \ v\in N_{r\theta}\mathcal{C}\mapsto r\theta+v.
\end{equation}
By composition we now obtain the required identification
\begin{equation}\label{eq:psicone}
\Psi_{\mathcal{C}}:\mathcal{U}\subset T^*\mathcal{C}\rightarrow\C^m,\ \
(\theta,r,\alpha_1+\alpha_2\,dr)\mapsto
r\theta+\tilde{J}(\alpha_2\theta+r^{-1}A_1).
\end{equation}
Now let $\alpha$ be a 1-form on $\mathcal{C}$. Then, under the above
identifications, $(\Psi_{\mathcal{C}}\circ\alpha)-\iota\simeq\alpha$. This shows
that if $\alpha\in C^\infty_{(\mu-1,\lambda-1)}(\mathcal{U})$ for some $\mu>2$,
$\lambda<2$ then $\Psi_{\mathcal{C}}\circ\alpha$ is a CS/AC submanifold in
$\C^m$ asymptotic to $\mathcal{C}$ with rate $(\mu,\lambda)$.
 
Notice also that
\begin{equation}\label{eq:psiequivariant}
\Psi_{\mathcal{C}}(\theta,tr,t^2\alpha_1+t\alpha_2\,dr)=t\Psi_{\mathcal{C}}
(\theta,r,\alpha_1+\alpha_2\,dr). 
\end{equation}
This suggests that we define an action of $\R^+$ on $T^*\mathcal{C}$ as follows:
\begin{equation}\label{eq:action}
\R^+\times T^*\mathcal{C}\rightarrow T^*\mathcal{C},\ \
t\cdot(\theta,r,\alpha_1+\alpha_2\,dr):=(\theta,tr,t^2\alpha_1+t\alpha_2\,dr).
\end{equation}
With respect to this action on $T^*\mathcal{C}$ and the standard action by
dilations on $\C^m$, Equation \ref{eq:psiequivariant} shows that
$\Psi_{\mathcal{C}}$ is an equivariant map. 

\begin{remark} \label{rem:action}
Equation \ref{eq:action} introduces an action on $T^*\mathcal{C}$ which rescales
both the base space and the fibres. We can also obtain it as follows. On any
cotangent bundle $T^*L$ there is a natural action 
\begin{equation*}
 \R^+\times T^*L\rightarrow T^*L,\ \ t\cdot(x,\alpha):=(x,t^2\alpha).
\end{equation*}
The induced action on 1-forms is such that, for the tautological 1-form
$\hat{\lambda}$, $t^*\hat{\lambda}=t^2\hat{\lambda}$. 

When $L=\Sigma\times (0,\infty)$ there is also a natural action
\begin{equation*}
 \R^+\times L\rightarrow L,\ \ t\cdot(\theta,r):=(\theta,tr).
\end{equation*}
This induces an action on $T^*L$ as follows:
\begin{equation*}
 \R^+\times T^*(\Sigma\times (0,\infty))\rightarrow
T^*(\Sigma\times(0,\infty)),\ \
t\cdot(\theta,r,\alpha_1+\alpha_2\,dr):=(\theta,tr,\alpha_1+t^{-1}\alpha_2\,dr).
\end{equation*}
The induced action on 1-forms preserves
$\hat{\lambda}:t^*\hat{\lambda}=\hat{\lambda}$. Equation \ref{eq:action}
coincides with the composed action and thus satisfies
$t^*\hat{\lambda}=t^2\hat{\lambda}$, so $t^*\hat{\omega}=t^2\hat{\omega}$.
\end{remark}

We now want to investigate the symplectic properties of the map
$\Psi_{\mathcal{C}}$. Let $\tilde{\omega}$ denote the standard symplectic
structure on $\C^m$. Since $\mathcal{C}$ is Lagrangian, the fibres of the normal
bundle define (locally) a Lagrangian foliation of $\C^m$, transverse to
$\mathcal{C}$. Using the fact that $\Psi_{\mathcal{C}}$ is the identity on
$\mathcal{C}$ and is linear on each fibre, one can check that, at each point of
$\mathcal{C}$, $(\Psi_{\mathcal{C}})^*\tilde{\omega}=\hat{\omega}$. Notice also
that $\Psi_{\mathcal{C}}$ identifies the foliation of $\C^m$ with the foliation
of $T^*\mathcal{C}$ defined by the fibres.

As in the proof of Theorem \ref{th:nbd_weinstein}, we now want to perturb
$\Psi_{\mathcal{C}}$ so as to obtain a local symplectomorphism
$\mathcal{U}\subset T^*\mathcal{C}\rightarrow \C^m$. As in that case, the idea
is to build a (local)  diffeomorphism $k:T^*\mathcal{C}\rightarrow
T^*\mathcal{C}$ such that $k_*=Id$ at each point of $\mathcal{C}$ and
$k^*(\Psi_{\mathcal{C}})^*\tilde{\omega}=\hat{\omega}$. The construction of such
$k$ is sufficiently explicit in \cite{weinstein} p. 333 to allow us to prove
that $k$ is equivariant with respect to the $\R^+$-action. Furthermore, the fact
that the fibres of $T^*\mathcal{C}$ are Lagrangian for both symplectic forms
implies that $k$ preserves these fibres, see \cite{weinstein} Theorem 7.1 for
details. Now define 
\begin{equation}\label{eq:lagconemap}
\Phi_{\mathcal{C}}:=\Psi_{\mathcal{C}}\circ k:\mathcal{U}\subset
T^*\mathcal{C}\rightarrow \C^m. 
\end{equation}
By construction, $\Phi_{\mathcal{C}}$ satisfies
$(\Phi_{\mathcal{C}})^*\tilde{\omega}=\hat{\omega}$. Furthermore,
$\Phi_{\mathcal{C}}$ is equivariant and its fibrewise linearization at each
$x\in\mathcal{C}$ coincides with $\Psi_{\mathcal{C}}$. Thus
$\Phi_{\mathcal{C}}=\Psi_{\mathcal{C}}+R$, for some $R$ satisfying 
\begin{equation}
|R(\theta,1,\alpha_1,\alpha_2)|=O(|\alpha_1|_{g'}^2+|\alpha_2|^2),\ \ \mbox{ as }|a_1|_{g'}+|a_2|\rightarrow 0. 
\end{equation}
Clearly $R$ is also equivariant. Thus
\begin{equation}\label{eq:Requivariance}
|R(\theta,t,\alpha_1,\alpha_2)|=|R(\theta,t\cdot
1,t^2t^{-2}\alpha_1,tt^{-1}\alpha_2)|=t\cdot
O(t^{-4}|\alpha_1|_{g'}^2+t^{-2}|\alpha_2|^2).
\end{equation}
The equivariance of R can be used to determine its asymptotic
behaviour with respect to $r$ after composition with 1-forms on $\mathcal{C}$. For example, given
any $\mu>2$ and $\lambda<2$, choose $\alpha$ in the space
$C^\infty_{(\mu-1,\lambda-1)}(\mathcal{U})$. Notice that, as
$r\rightarrow\infty$, $r^{-1}|\alpha_1|_{g'}=|\alpha_1|_g=O(r^{\lambda-1})$.
This implies $r^{-4}|\alpha_1|_{g'}^2=O(r^{2\lambda-4})$. Analogously,
$|\alpha_2|=O(r^{\lambda-1})$ so $r^{-2}|\alpha_2|^2=O(r^{2\lambda-4})$.
Equation \ref{eq:Requivariance} then shows that 
$(R\circ\alpha)(\theta,r)=R(\theta,r,\alpha(\theta,r))$ satisfies $|R\circ\alpha|=O(r^{2\lambda-3})$ as $r\rightarrow \infty$. Further calculations show that
the derivatives of $R\circ\alpha$ scale correspondingly, \textit{e.g.} 
\begin{equation}\label{eq:Rscaling}
|(R\circ\alpha)_*(\partial r)|=O(r^{2\lambda-4}),\ \
|(R\circ\alpha)_*(r^{-1}\partial\theta_i)|=O(r^{2\lambda-4}).
\end{equation}
More generally, $|\tnabla^k(R\circ\alpha)|=O(r^{2\lambda-3-k})$. As a result,
\begin{equation*}
 \begin{split}
  |\tnabla^k(\Phi_{\mathcal{C}}\circ\alpha-\iota)| &= |\tnabla^k((\Psi_{\mathcal{C}}+R)\circ\alpha-\iota)|
\leq |\tnabla^k(\Psi_{\mathcal{C}}\circ\alpha-\iota)|+|\tnabla^k(R\circ\alpha)|\\
&= O(r^{\lambda-1-k})+O(r^{2\lambda-3-k})=O(r^{\lambda-1-k}),
 \end{split}
\end{equation*}
where we use $\lambda<2$. This shows that $\Phi_{\mathcal{C}}\circ \alpha$ is a CS/AC
Lagrangian submanifold asymptotic to $\mathcal{C}$ with rate $(\mu,\lambda)$.
Conversely, one can show that any Lagrangian submanifold $L$ of $\C^m$ which
admits a parametrization which is $C^1$-close to $\iota$ and which is asymptotic
to $\iota$ in the sense of Equations \ref{eq:aclagdecay} and \ref{eq:cslagdecay}
corresponds to a closed 1-form $\alpha\in
C^\infty_{(\mu-1,\lambda-1)}(\mathcal{U})$.

In complete analogy with Section \ref{ss:cptlagdefs} we can use
$\Phi_{\mathcal{C}}$ and the closed forms in the space
$C^\infty_{(\mu-1,\lambda-1)}(\mathcal{U})$ to define a topology on the set of
Lagrangian submanifolds which admit a parametrization
$\iota:\Sigma\times\R^+\rightarrow\C^m$ which is asymptotic to $\mathcal{C}$
with rate $(\mu,\lambda)$. The connected component containing $\mathcal{C}$
defines the \textit{moduli space of CS/AC Lagrangian deformations of
$\mathcal{C}$ with rate $(\mu,\lambda)$}. 

We conclude with a last comment on the differential properties of
$\Psi_{\mathcal{C}}$. Recall the following general fact.
\begin{lemma} \label{l:diffvsconn}
Let $E\rightarrow M$ be a vector bundle, endowed with a connection $\nabla$. Let
$\sigma:M\rightarrow E$ be a section of $E$. Choose $v\in T_pM$. The connection
defines a decomposition into ``vertical" and ``horizontal" subspaces 
\begin{equation}
T_{\sigma(p)}E=V_{\sigma(p)}\oplus H_{\sigma(p)}, \mbox{\ \ with\ \
}V_{\sigma(p)}\simeq E_p,\ H_{\sigma(p)}\simeq T_p(M).
\end{equation} 
Under these identifications, $\sigma_*(v)\simeq \nabla_v\sigma+v$. 
\end{lemma}
We can apply Lemma \ref{l:diffvsconn} as follows. Let $\alpha$ be a section of
$T^*\mathcal{C}$ so that $\Psi_{\mathcal{C}}\circ\alpha:\mathcal{C}\rightarrow
\C^m$ is a submanifold of $\C^m$. Choose $v\in T_{r\theta}\mathcal{C}$. Then,
using the identifications \ref{eq:cotg=tg}, \ref{eq:tg=normal},
\ref{eq:normal=ambient} and Lemma \ref{l:diffvsconn}, 
\begin{equation}\label{eq:psi_*}
(\Psi_{\mathcal{C}}\circ\alpha)_*(v)\simeq \tnabla_v \alpha+v,
\end{equation}
where $\tnabla$ denotes the Levi-Civita connection on $T^*\mathcal{C}$.


\subsection{Third case: CS/AC Lagrangian submanifolds in $\C^m$}\label{ss:accslagdefs}

Let $\iota:L\rightarrow\C^m$ be an AC Lagrangian submanifold with rate
$\boldsymbol{\lambda}$, centers $p_i$ and ends $S_i$. Using the notation of Section
\ref{ss:conelagdefs}, 
the map $\Phi_{\mathcal{C}_i}+p_i:T^*\mathcal{C}_i\rightarrow\C^m$ identifies $\iota(S_i)\subset\C^m$ with the graph $\Gamma(\alpha_i)$ of some closed 1-form $\alpha_i$. This construction also determines a distinguished coordinate
system $\phi_i$ by imposing the relation
\begin{equation*}
 \phi_i:\mathcal{C}_i\rightarrow S_i, \ \ \iota\circ\phi_i=\Phi_{\mathcal{C}_i}\circ\alpha_i. 
\end{equation*}
Letting $(d\phi_i)^*:T^*S_i\rightarrow T^*\mathcal{C}_i$
denote the corresponding identification of cotangent bundles, we obtain an
identification of the zero section $\mathcal{C}_i$ with the zero section $S_i$.
We can use the symplectomorphism $\tau_{\alpha_i}$ defined in Equation
\ref{eq:translation} to ``bridge the gap'' between these identifications,
obtaining a symplectomorphism
\begin{equation}
\Phi_{S_i}:\mathcal{U}_i\subset T^*S_i\rightarrow \C^m, \ \
\Phi_{S_i}:=\Phi_{\mathcal{C}_i}\circ\tau_{\alpha_i}\circ (d\phi_i)^*+p_i
\end{equation}
which restricts to the identity on $S_i$. These maps provide a Lagrangian
neighbourhood for each end of $L$. Using the same methods as in the proof of Theorem
\ref{th:nbd_weinstein} one can interpolate between these maps. The final result
is a symplectomorphism
\begin{equation}\label{eq:aclagmap}
\Phi_L:\mathcal{U}\subset T^*L\rightarrow \C^m
\end{equation}
which restricts to the identity along $L$. This allows us to parametrize AC
deformations of $L$ with rate $\boldsymbol{\lambda}$ in terms of closed 1-forms
in the space $C^\infty_{\boldsymbol{\lambda}-1}(\mathcal{U})$.

More generally, given a CS or CS/AC Lagrangian submanifold $L$ in $\C^m$, the same
ideas define a symplectomorphism $\Phi_{L}$ as in Equation \ref{eq:aclagmap}. The same is true for a CS
submanifold in $M$: this time it is necessary to insert appropriate compositions
by $\Upsilon_i$. We refer to Joyce \cite{joyce:I} for additional details
concerning constructions of this type. 

Coupling these results with Decompositions \ref{decomp:closedforms_growth},
\ref{decomp:closedforms_decay} and \ref{decomp:closedforms_accs} now gives a
good idea of the local structure of the corresponding moduli spaces of
Lagrangian deformations, defined as in Sections \ref{ss:cptlagdefs} and
\ref{ss:conelagdefs}.


\subsection{Lagrangian deformations with moving
singularities}\label{ss:movingsings}

In Section \ref{ss:accslagdefs} the given Lagrangian submanifold $L$ is deformed
keeping the singular points fixed in the ambient manifold $\C^m$ or $M$. It is
also natural to want to deform $L$ allowing the singular points to move within
the ambient space. Analogously, one might want to allow the corresponding
Lagrangian cones $\cone_i$ to rotate in $\C^m$. The correct set-up for doing
this when $\iota:L\rightarrow M$ is a CS Lagrangian submanifold with singularities
$\{x_1,\dots,x_s\}$ and identifications $\upsilon_i$ is as follows. The ideas are based on \cite{joyce:II}
Section 5.1. Define
\begin{equation}
P:=\{(p,\upsilon):p\in M,\ \upsilon:\C^m\rightarrow T_pM \mbox{ such that
}\upsilon^*\omega=\tilde{\omega}, \upsilon^*J=\tilde{J}\}.
\end{equation}
$P$ is a $\unitary m$-principal fibre bundle over $M$ with the action 
$$\unitary m\times P\rightarrow P,\ \ M\cdot (p,\upsilon):=(p,\upsilon\circ
M^{-1}).$$
As such, $P$ is a smooth manifold of dimension $m^2+2m$.

Our aim is to use one copy of $P$ to parametrize the location of each singular
point $p_i=\iota(x_i)\in M$ and the direction of the corresponding cone $\cone_i\subset
\C^m$: the group action will allow the cone to rotate leaving the singular point
fixed. As we are interested only in small deformations of $L$ we can restrict
our attention to a small open neighbourhood of the pair $(p_i,\upsilon_i)\in P$. In general the
$\cone_i$ will have some symmetry group $G_i\subset \unitary m$, \textit{i.e.}
the action of this $G_i$ will leave the cone fixed. To ensure that we have no
redundant parameters we must therefore further restrict our attention to a
\textit{slice} of our open neighbourhood, \textit{i.e.} a smooth submanifold
transverse to the orbits of $G_i$. We denote this slice $\E_i$: it is a subset
of $P$ containing $(p_i,\upsilon_i)$ and of dimension $m^2+2m-\mbox{dim}(G_i)$.
We then set $\E:=\E_1\times\dots\times\E_s$. The point
$e:=(p_1,\upsilon_1),\dots,(p_s,\upsilon_s))\in \E$ will denote the initial
data as in Definition \ref{def:generalcslagsub}.

We now want to extend the datum of $(L,\iota)$ to a family of Lagrangian submanifolds
$(L,\iota_{\tilde{e}})$ parametrized by
$\tilde{e}=((\tilde{p}_1,\tilde{\upsilon}_1),\dots,(\tilde{p}_s,\tilde{\upsilon}_s))\in
\E$ (making $\E$ smaller if necessary). Each $(L,\iota_{\tilde{e}})$ should satisfy $\iota_{\tilde{e}}(p_i)=\tilde{p_i}$ and admit identifications $\tilde{\upsilon}_i$ and cones $\mathcal{C}_i$
as in Definition \ref{def:generalcslagsub}. We further require that $\iota_e=\iota$ globally and that $\iota_{\tilde{e}}=\iota$ outside a neighbourhood of the singularities.
The construction of such a family is actually straight-forward: using the maps
$\Upsilon_i$, it reduces to a choice of an appropriate family of
compactly-supported symplectomorphisms of $\C^m$.

It is now possible to choose an open neighbourhood $\mathcal{U}\subset T^*L$ and
embeddings $\Phi_{L}^{\tilde{e}}:\mathcal{U}\rightarrow M$ which, away from the
singularities, coincide with the embedding $\Phi_{L}$ introduced in Section
\ref{ss:accslagdefs}. The final result is that, after such a choice, the
\textit{moduli space of CS Lagrangian deformations of $L$ with rate
$\boldsymbol{\mu}$ and moving singularities} can be parametrized in terms of
pairs $(\tilde{e}, \alpha)$ where $\tilde{e}\in\E$ and $\alpha$ is a closed 1-form
on $L$ belonging to the space $C^\infty_{\boldsymbol{\mu}-1}(\mathcal{U})$. 

Analogous results hold of course for CS and CS/AC submanifolds in $\C^m$. In
this case it is sufficient to set $P:=\{(p,\upsilon)\}$, with $p\in\C^m$ and
$\upsilon\in \unitary m$.

\subsection{Other convergence rates}\label{ss:otherrates}
The previous sections discuss the deformation theory of a Lagrangian conifold $(L,\iota)$ with convergence rate $(\boldsymbol{\mu}, \boldsymbol{\lambda})$ within the class of deformations which preserve the convergence rate. For some purposes, cf. \cite{pacini:slgluing}, it may also be useful to consider other deformation classes, obtained via closed 1-forms in the space $C^\infty_{\boldsymbol{\beta}-1}(\mathcal{U})$, for some other weight $\boldsymbol{\beta}$. We consider here two cases.

The first case is when $2<\beta_i\leq \mu_i$ on each CS end and $2>\beta_i\geq\lambda_i$ on each AC end: in other words, we relax the convergence rate of the deformed submanifolds. This case is simple: the initial conifold has \textit{a fortiori} convergence rate $\boldsymbol{\beta}$, so the above theory immediately shows that the closed forms in $C^\infty_{\boldsymbol{\beta}-1}(\mathcal{U})$ parametrize all other Lagrangian conifolds with this rate. 

The second case is when $2<\mu_i<\beta_i$ on each CS end and $2>\lambda_i>\beta_i$ on each AC end: in other words, we strengthen the convergence rate. In this case the closed 1-forms in $C^\infty_{\boldsymbol{\beta}-1}(\mathcal{U})$ parametrize the Lagrangian immersions $(L,\iota')$ which are asymptotic to $(L,\iota)$ in a sense analogous to Definitions \ref{def:aclagsub}, \ref{def:cslagsub}: on each AC end, up to appropriate diffeomorphisms $\phi_i$ and for $r\rightarrow\infty$,
\begin{equation*}\label{eq:generalizedaclagdecay}
|\tnabla^k(\iota'-\iota)|=O(r^{\beta_i-1-k})
\end{equation*}
and on each CS end, for $r\rightarrow 0$,
\begin{equation*}\label{eq:generalizedaclagdecay}
|\tnabla^k(\iota'-\iota)|=O(r^{\beta_i-1-k}).
\end{equation*}
To prove this, as in Section \ref{ss:accslagdefs}, assume $\iota$ is obtained as the graph of a 1-form $\alpha$ so that on each end $\iota=\Phi_{\mathcal{C}}\circ\alpha$ and $\Phi_L=\Phi_{\mathcal{C}}\circ\tau_\alpha$. Choose a closed 1-form $\alpha'\in C^\infty_{\boldsymbol{\beta}-1}(\mathcal{U})$ and set $\iota':=\Phi_L\circ\alpha'=\Phi_{\mathcal{C}}\circ(\alpha+\alpha')$. Then
\begin{equation*}
\begin{split}
 |\iota'-\iota|&=|\Phi_\mathcal{C}\circ(\alpha+\alpha')-\Phi_{\mathcal{C}}\circ\alpha|=|(\Psi_{\mathcal{C}}+R)\circ(\alpha+\alpha')-(\Psi_{\mathcal{C}}+R)\circ\alpha|\\
&\leq|\Psi_{\mathcal{C}}\circ\alpha'|+|R\circ(\alpha+\alpha')-R\circ\alpha|\\
&=O(r^{\beta-1})+O(r^{\beta-2+\lambda-2+1})=O(r^{\beta-1}),\\
\end{split}
\end{equation*}
 where we use the fact that $R(\theta,1,\cdot)$ is roughly quadratic in the $\cdot$ variable, so 
\begin{equation*}
|R(\theta,1,\alpha'+\alpha)-R(\theta,1,\alpha)|=O(|\alpha'|\cdot|\alpha|).
\end{equation*}
 We then conclude via the reasoning already described following Equation \ref{eq:Requivariance}. 

Similar calculations give estimates on the derivatives.


\section{Special Lagrangian conifolds}\label{s:slgeometry}
\begin{definition}\label{def:cy} A \textit{Calabi-Yau} (CY) manifold is the data
of a K\"ahler manifold ($M^{2m}$,$g$,$J$,$\omega$) and a non-zero ($m,0$)-form
$\Omega$ satisfying $\nabla\Omega\equiv 0$ and normalized by the condition
$\omega^m/m!=(-1)^{m(m-1)/2}(i/2)^m\Omega\wedge\bar{\Omega}$.

In particular $\Omega$ is holomorphic and the holonomy of $(M,g)$ is contained
in $\sunitary m$. We will refer to $\Omega$ as the \textit{holomorphic volume
form} on $M$.
\end{definition}

\begin{example}\label{e:C^m}
The simplest example of a CY manifold is $\C^m$ with its standard structures
$\tilde{g}$, $\tilde{J}$, $\tilde{\omega}$ and
$\tilde{\Omega}:=dz^1\wedge\dots\wedge dz^m$.
\end{example}
\begin{definition}\label{def:sl}
Let $M^{2m}$ be a CY manifold and $L^m\rightarrow M$ be an immersed or embedded
Lagrangian submanifold. We can restrict $\Omega$ to $L$, obtaining a
non-vanishing complex-valued $m$-form $\Omega_{|L}$ on $L$. We say that $L$ is
\textit{special Lagrangian} (SL) iff this form is real, \textit{i.e.}
$\Imag\,\Omega_{|L}\equiv 0$. In this case $\Real\,\Omega_{|L}$ defines a volume
form on $L$, thus a natural orientation.
\end{definition}
Lagrangian submanifolds (especially the immersed ones) tend to be very ``soft"
objects: for example, Section \ref{s:lagdefs} shows that they have
infinite-dimensional moduli spaces. They also easily allow for cutting, pasting
and desingularization procedures. The ``special" condition rigidifies them
considerably: the corresponding deformation, gluing and desingularization
processes require much ``harder" techniques. Cf. \textit{e.g.}
\cite{haskinskapouleas}, \cite{joyce:III}, \cite{joyce:IV}, \cite{pacini:slgluing} for recent gluing
results and \cite{haskinspacini} for local desingularization issues.

\begin{definition} \label{def:accssl}
We can define AC, CS and CS/AC special Lagrangian submanifolds in $\C^m$ exactly
as in Definitions \ref{def:aclagsub}, \ref{def:cslagsub} and
\ref{def:accslagsub}, simply adding the requirement that the submanifolds be
special Lagrangian. In particular this implies that the cones $\mathcal{C}_i$
are SL in $\C^m$. Following Definition \ref{def:generalcslagsub} we can also
define CS special Lagrangian submanifolds in a general CY manifold $M$: in this
case it is necessary to also add the requirement that
$\upsilon_i^*\Omega=\tilde{\Omega}$.

We use the generic term \textit{special Lagrangian conifold} to refer to any of
the above.
\end{definition}

\begin{remark} \label{rem:muregularity} 
It follows from Joyce \cite{joyce:I} Theorem 5.5 that if $L$ is a CS or CS/AC SL
submanifold with respect to some rate $\boldsymbol{\mu}=2+\epsilon$ with
$\epsilon$ in a certain range $(0,\epsilon_0)$ then it is also CS or CS/AC with
respect to any other rate of the form $\boldsymbol{\mu}'=2+\epsilon'$ with
$\epsilon'\in (0,\epsilon_0)$. The precise value of $\epsilon_0$ is determined
by certain \textit{exceptional weights} for the cones $\mathcal{C}_i$,
introduced in Section \ref{s:reviewlaplace}. We refer to \cite{joyce:I} for
details.
\end{remark}
  
\begin{example} \label{e:slcone}
Let $\mathcal{C}$ be a Lagrangian cone in $\C^m$ with smooth link
$\Sigma^{m-1}$. It can be shown that $\mathcal{C}$ is SL (with respect to some
holomorphic volume form $e^{i\theta}\tilde{\Omega}$) iff $\Sigma$ is minimal in
$\Sph^{2m-1}$ with respect to the natural metric on the sphere. Then
$\mathcal{C}$ is a CS/AC SL in $\C^m$. Cf. \textit{e.g.} \cite{harveylawson},
\cite{haskins}, \cite{haskinskapouleas}, \cite{haskinspacini}, \cite{joyce:symmetries} for examples.

We refer to Joyce \cite{joyce:V} Section 6.4 for examples of AC SLs in $\C^m$
with various rates.
\end{example}


\section{Setting up the SL deformation problem}\label{s:setup}
If $\iota:L\rightarrow M$ is a SL conifold we can specialize the framework of Section
\ref{s:lagdefs} to study the SL deformations of $L$. Notice that the SL
condition is again invariant under reparametrizations. Thus, if $L$ is smooth
and compact, the \textit{moduli space} $\mathcal{M}_L$ \textit{of SL
deformations of $L$} can be defined as the connected component containing $L$ of
the subset of SL submanifolds in $\mbox{Lag}(L,M)/\mbox{Diff}(L)$. As seen in
Sections \ref{ss:conelagdefs} and \ref{ss:accslagdefs}, if $L$ is an AC, CS or
CS/AC Lagrangian submanifold with specific rates of growth/decay on the ends, we can obtain
moduli spaces of Lagrangian or SL deformations of $L$ with those same rates by
simply restricting our attention to closed 1-forms on $L$ which satisfy
corresponding growth/decay conditions. 

Our ultimate goal is to prove that moduli spaces of SL conifolds often admit a
natural smooth structure with respect to which they are finite-dimensional
manifolds. Failing this, we want to identify the obstructions which prevent this
from happening. Generally speaking, the strategy for proving these results will
be to view $\mathcal{M}_L$ locally as the zero set of some smooth map $F$
defined on the space of closed forms in $C^\infty(\mathcal{U})$ (when $L$ is
smooth and compact) or in
$C^\infty_{(\boldsymbol{\mu}-1,\boldsymbol{\lambda}-1)}(\mathcal{U})$ (when $L$
is CS/AC with rate $(\boldsymbol{\mu},\boldsymbol{\lambda})$): we can then
attempt to use the Implicit Function Theorem to prove that this zero set
is smooth. 

The choice of $F$ is dictated by Definition \ref{def:sl}. Let $\Omega$ denote
the given holomorphic volume form on $M$. Then $F$ must compute the values
of $\Imag\,\Omega$ on each Lagrangian deformation of $L$. In the following
sections we present the precise construction of $F$ and study its properties,
for each case of interest.

\textit{Note: }To simplify the notation, from now on we will drop the immersion $\iota:L\rightarrow M$ and simply identify $L$ with its image. In particular we will identify the singularities $x_i$ with their images $\iota(x_i)$.


\subsection{First case: smooth compact special Lagrangians}\label{ss:cptslsetup}

Let $L\subset M$ be a smooth compact SL submanifold, endowed with the induced
metric $g$ and orientation. Define $\Phi_L:\mathcal{U}\rightarrow M$ as in
Section \ref{ss:cptlagdefs}. Consider the pull-back real $m$-form
$\Phi_L^*(\Imag\,\Omega)$ defined on $\mathcal{U}$. Given any closed $\alpha\in
C^\infty(\mathcal{U})$, let $\Gamma(\alpha)$ denote the submanifold in
$\mathcal{U}$ defined by its graph. It is diffeomorphic to $L$ via the
projection $\pi:T^*L\rightarrow L$. The pull-back form restricts to an $m$-form
$\Phi_L^*(\Imag\,\Omega)_{|\Gamma(\alpha)}$ on $\Gamma(\alpha)$. It is clear
from Definition \ref{def:sl} that $\Gamma(\alpha)$ is SL iff this form vanishes.
We can now pull this form back to $L$ via $\alpha$ (equivalently, push it down to $L$ via $\pi_*$), obtaining a real $m$-form on
$L$: then $\Gamma(\alpha)$ is SL iff this
form vanishes on $L$. Finally, let $\star$ denote the \textit{Hodge
star operator} defined on $L$ by $g$ and the orientation. Using this operator we
can reduce any $m$-form on $L$ to a function. 

Summarizing, let $\mathcal{D}_L$ denote the space of closed 1-forms on $L$ whose graph lies
in $\mathcal{U}$. We then define the map $F$ as follows. 
\begin{equation}\label{eq:defF}
F:\mathcal{D}_L\rightarrow C^\infty (L),\ \ \alpha\mapsto
\star(\alpha^*(\Phi_L^*\Imag\,\Omega))=\star((\Phi_L\circ\alpha)^*\Imag\,\Omega).
\end{equation}
\begin{prop} \label{prop:cptnonlinear}
The non-linear map $F$ has the following properties:
\begin{enumerate}
\item The set $F^{-1}(0)$ parametrizes the space of all SL deformations of $L$
which are $C^1$-close to $L$.
\item $F$ is a smooth map between Fr\'echet spaces.  Furthermore, for each
$\alpha\in\mathcal{D}_L$, $\int_LF(\alpha)\,vol_g=0$.
\item The linearization $dF[0]$ of $F$ at $0$ coincides with the operator $d^*$,
\textit{i.e.} 
\begin{equation}\label{eq:linearization}
dF[0](\alpha)=d^*\alpha.
\end{equation}
\end{enumerate}
\end{prop}
\begin{proof} These results are standard, cf. \cite{mclean} or \cite{joyce:I}
Prop. 2.10. However for the reader's convenience we give a sketch of the
argument with respect to our own set of conventions. To simplify the notation we
identify $\mathcal{U}$ with its image in $M$ via $\Phi_L$. This allows us to
write 
\begin{equation}
F(\alpha)=\star(\pi_*(\Imag\,\Omega_{|\Gamma(\alpha)})).
\end{equation} 
We also identify $L$ with the zero section in $T^*L$. 

The first statement follows directly from the definition of $F$ and the results
of Section \ref{ss:cptlagdefs}. More precisely the statement is that, up to
composition with $\Phi_L$, $F^{-1}(0)$ coincides with the set of SL submanifolds
which admit a parametrization which is $C^1$-close to some parametrization of
$L$.

To prove the second statement, notice that
$\int_LF(\alpha)\,vol_g=\int_{\Gamma(\alpha)}\Imag\,\Omega$. The fact that
$\Omega$ is closed implies that $\Imag\,\Omega$ is closed. Furthermore the
submanifold $\Gamma(\alpha)$ is homotopic, thus homologous, to the zero section
$L$. Thus $\int_{\Gamma(\alpha)}\Imag\,\Omega=\int_L\Imag\,\Omega=0$ because $L$
is SL. The smoothness of $F$ is clear from its definition. 

To prove Equation \ref{eq:linearization}, fix any $\alpha\in \Lambda^1(L)$ and
let $v$ denote the normal vector field along $L$ determined by imposing
$\alpha(\cdot)\equiv\omega(v,\cdot)$. We can extend $v$ to a global vector field
$v$ on $M$. Let $\phi_s$ denote any 1-parameter family of diffeomorphisms of $M$
such that $d/ds(\phi_s(x))_{|s=0}=v(x)$. Then the two 1-parameter families of
$m$-forms on $L$, $(s\alpha)^*(\Imag\,\Omega)=\pi_*(\Imag \,\Omega_{|\Gamma(s\alpha)})$ and
$(\phi_s^*\Imag\,\Omega)_{|L}$, coincide up to first order so that standard
calculus of Lie derivatives shows that
\begin{eqnarray*}
dF[0](\alpha)\,vol_g &=& d/ds
(F(s\alpha)\,vol_g)_{|s=0}\\
&=& d/ds(\phi_s^*\Imag\,\Omega)_{|L;\,s=0}\\
&=& (\mathcal{L}_v\Imag\,\Omega)_{|L}= (di_v\Imag\,\Omega)_{|L},
\end{eqnarray*}
where in the last equality we use \textit{Cartan's formula}
$\mathcal{L}_v=di_v+i_vd$ and the fact that $\Imag\,\Omega$ is closed.

We now claim that $(i_v\Imag\,\Omega)_{|L}\equiv-\star\alpha$ on $L$. This is a
linear algebra statement so we can check it point by point. We can also assume
that $v$ is a unit vector at that point. Fix a point $x\in L$ and an isomorphism
$T_xM\simeq \C^m$ identifying the CY structures on $T_xM$ with the standard
structures on $\C^m$. This map will identify $T_xL$ with a SL $m$-plane $\Pi$ in
$\C^m$. Consider the action of $\sunitary m$ on the Grassmannian of $m$-planes
in $\C^m$. In \cite{harveylawson} page 89 it is shown that $\sunitary m$ acts
transitively on the subset of SL $m$-planes and that the isotropy subgroup
corresponding to the distinguished SL plane $\R^m:=\mbox{span}\{\partial
x^1,\dots,\partial x^m\}$ is $\sorth m\subset \sunitary m$; in other words, the
set of SL $m$-planes in $\C^m$ can be identified with the homogeneous space
$\sunitary m/\sorth m$. Up to a rotation in $\sunitary m$ we can assume that
$\Pi=\R^m$. Up to a rotation in $\sorth m$ we can further assume that
$v(x)=\partial y^1$. It is thus sufficient to check our claim in this case only.
We can write $\Imag\,\Omega=dy^1\wedge dx^2\wedge\dots\wedge dx^m+(\dots)$. It
follows that $(i_v\Imag\,\Omega)_{|\R^m}=dx^2\wedge\dots\wedge dx^m$. On the
other hand $\alpha=-dx^1$, proving the claim, thus Equation
\ref{eq:linearization}.
\end{proof}

\begin{remark} \label{rem:cptQ}
Notice that $\star$ depends on $x\in L$, $\Gamma(\alpha)$ depends on $\alpha$
and $\Phi_L^*\Imag\,\Omega_{|\Gamma(\alpha)}$ depends on $\alpha$ and
$\nabla\alpha$. We can thus think of $F$ as being obtained from an underlying
smooth function
\begin{equation}
F'=F'(x,y,z):\mathcal{U}\oplus(T^*L\otimes T^*L)\rightarrow \R
\end{equation}
via the following relationship: 
\begin{equation}\label{eq:F_unwrapped}
F(\alpha)=F'(x,\alpha(x),\nabla\alpha(x)).
\end{equation}
More specifically, $F'$ can be defined as follows. Choose a point $(x,y)\in
\mathcal{U}$. Let $e_1,\dots,e_m$ be an orthonormal positive basis of $T_xL$.
Now choose any $z\in T_x^*L\otimes T_x^*L$. Recall from Lemma \ref{l:diffvsconn}
that, using the Levi-Civita connection, $T_{(x,y)}\mathcal{U}\simeq T_x^*L\oplus
T_xL$. Thus the vectors $(i_{e_i}z,e_i)$ span an $m$-plane in
$T_{(x,y)}\mathcal{U}$; when $y=\alpha$ and $z=\nabla\alpha$, this $m$-plane
coincides with $T_{(x,\alpha)}\Gamma(\alpha)$. We can now define 
\begin{equation}\label{eq:F'}
F'(x,y,z):=\Phi_L^*\Imag\,\Omega_{|(x,y)}((i_{e_1}z,e_1),\dots,(i_{e_m}z,
e_m)).
\end{equation}
For any fixed $x\in L$, $y$ and $z$ vary in the linear space
$T^*_xL\oplus(T^*_xL\otimes T^*_xL)$ so Taylor's theorem shows
\begin{equation}
F'(x,y,z)=F'(x,0,0)+\frac{\partial F'}{\partial
y}(x,0,0)\,y+\frac{\partial F'}{\partial z}(x,0,0)\,z+Q'(x,y,z)
\end{equation}
for some smooth $Q'=Q'(x,y,z)$ satisfying
$Q'(x,y,z)=O(|y|^2+|z|^2)$ for each $x$, as $|y|\rightarrow 0$ and
$|z|\rightarrow 0$. By substitution we find
\begin{eqnarray*}
F(\alpha) &=& F'(x,\alpha(x),\nabla\alpha(x))\\
&=&F'(x,0,0)+\frac{\partial F'}{\partial
y}(x,0,0)\,\alpha(x)+\frac{\partial F'}{\partial
z}(x,0,0)\,\nabla\alpha(x)+Q'(x,\alpha(x),\nabla\alpha(x)).
\end{eqnarray*}
The fact that $L$ is SL implies that $F'(x,0,0)\equiv 0$. Notice also that
by the chain rule
\begin{equation*}
d/ds(F(s\alpha))_{|s=0}=d/ds(F'(x,s\alpha(x),s\nabla\alpha(x))_{|s=0}=\frac
{\partial F'}{\partial y}(x,0,0)\,\alpha(x)+\frac{\partial
F'}{\partial z}(x,0,0)\,\nabla\alpha(x).
\end{equation*}
On the other hand, $d/ds(F(s\alpha))_{|s=0}=dF[0](\alpha)=d^*\alpha$. 
Combining these equations leads to
\begin{equation}\label{eq:cptQ}
F(\alpha)=d^*\alpha+Q'(x,\alpha(x),\nabla\alpha(x)).
\end{equation}
 \end{remark}


\subsection{Second case: special Lagrangian cones in
$\C^m$}\label{ss:coneslsetup}

Let $\mathcal{C}$ be a SL cone in $\C^m$, endowed with the induced metric
$\tilde{g}$ and orientation. Define $\Phi_{\mathcal{C}}:\mathcal{U}\rightarrow
\C^m$ as in Section \ref{ss:conelagdefs}. Fix any $\mu>2$, $\lambda<2$. Let
$\mathcal{D}_{\mathcal{C}}$ denote the space of closed 1-forms in
$C^\infty_{(\mu-1,\lambda-1)}(\Lambda^1)$ whose graph lies in $\mathcal{U}$.
Given $\alpha\in \mathcal{D}_{\mathcal{C}}$, define $F(\alpha)$ as in Equation
\ref{eq:defF}.

\begin{prop} \label{prop:conenonlinear}
The non-linear map $F$ has the following properties:
\begin{enumerate}
\item The set $F^{-1}(0)$ parametrizes the space of all SL deformations of
$\mathcal{C}$ which are $C^1$-close to $L$ and are asymptotic to $\mathcal{C}$
with rate $(\mu,\lambda)$.
\item $F$ is a well-defined smooth map
\begin{equation*}
F:\mathcal{D}_{\mathcal{C}}\rightarrow
C^\infty_{(\mu-2,\lambda-2)}(\mathcal{C}).
\end{equation*}
In particular, for each $\alpha\in\mathcal{D}_{\mathcal{C}}$, $F(\alpha)\in
C^\infty_{(\mu-2,\lambda-2)}(\mathcal{C})$.
\item The linearization $dF[0]$ of $F$ at $0$ coincides with the operator $d^*$,
\textit{i.e.} 
\begin{equation}
dF[0](\alpha)=d^*\alpha.
\end{equation}
\end{enumerate}
\end{prop}
\begin{proof}
The first statement follows from the definition of $F$ and the results of
Section \ref{ss:conelagdefs}. Concerning the second statement, we may write
\begin{eqnarray*}
F(\alpha) &=&
\star(\alpha^*(\Phi_{\mathcal{C}}^*\Imag\,\tilde{\Omega}))=\Imag\,\tilde{\Omega}((\Phi_{\mathcal{C}}\circ\alpha)_*(e_1),\dots,(\Phi_{
\mathcal{C}}\circ\alpha)_*(e_m))\\
&=&
\Imag\,\tilde{\Omega}((\Psi_{\mathcal{C}}
\circ\alpha)_*(e_1)+(R\circ\alpha)_*(e_1),\dots,(\Psi_{\mathcal{C}}
\circ\alpha)_*(e_m)+(R\circ\alpha)_*(e_m))\\
&=&
\Imag\,\tilde{\Omega}((\Psi_{\mathcal{C}}\circ\alpha)_*(e_1),\dots,(\Psi_{
\mathcal{C}}\circ\alpha)_*(e_m))+\dots,\\
\end{eqnarray*}
where $e_i$ is a local $\tg$-orthornomal basis of $T\mathcal{C}$.

Consider this last equation as $r\rightarrow\infty$. Equation \ref{eq:psi_*}
shows that its first term is of the form
$\Imag\,\tilde{\Omega}(e_1,\dots,e_m)+O(r^{\lambda-2})$. The first term here
vanishes because $\mathcal{C}$ is SL, leaving the term $O(r^{\lambda-2})$.
Equation \ref{eq:Rscaling} shows that the remaining terms in $F(\alpha)$ are of
the form $O(r^{2\lambda-4})$. Analogous methods apply for $r\rightarrow 0$,
showing that $F(\alpha)\in C^0_{(\mu-2,\lambda-2)}(\mathcal{C})$. 

To study the derivatives of $F(\alpha)$ we endow $\mathcal{U}$ with the metric and Levi-Civita connection $\nabla$ pulled back from $\C^m$ via $\Phi_{\mathcal{C}}$, so that $\nabla(\Phi_{\mathcal{C}}^*\Imag\,\tilde{\Omega})=\Phi_{\mathcal{C}}^*(\tnabla\Imag\,\tilde{\Omega})=0$. Let $g$ denote the induced metric on $\Gamma(\alpha)$. Then $\mathcal{C}$ can be endowed with either the metric $\tg$ and induced connection $\tnabla$ or with the metric $\alpha^*g$ and induced connection $\nabla$. One can check, or cf. \cite{pacini:weighted}, that the fact that $\alpha^*g$ is asymptotic to $\tg$ implies that the difference tensor $A:=\nabla-\tnabla$ satisfies $|A|=O(r^{\lambda-3})$, as $r\rightarrow \infty$. Notice that
\begin{equation*}
 F(\alpha)\,\mbox{vol}_{\tilde{g}}=(\Phi_{\mathcal{C}}\circ\alpha)^*\Imag\,\tilde{\Omega}
\end{equation*}
so, taking derivatives,
\begin{equation*}
 \nabla(F(\alpha)\,\mbox{vol}_{\tilde{g}})=\nabla((\Phi_{\mathcal{C}}\circ\alpha)^*\Imag\,\tilde{\Omega})=(\Phi_{\mathcal{C}}\circ\alpha)^*(\tnabla\Imag\,\tilde{\Omega})=0.
\end{equation*}
This implies
\begin{equation*}
 |(\nabla F(\alpha))\otimes\mbox{vol}_{\tg}|=|F(\alpha)\cdot\nabla(\mbox{vol}_{\tg})|=O(r^{\lambda-2})|\nabla(\mbox{vol}_{\tg})|.
\end{equation*}
Write $\mbox{vol}_{\tg}=e_1^*\otimes\dots\otimes e_m^*$ so that $\nabla(\mbox{vol}_{\tg})=\nabla e_1^*\otimes\dots\otimes e_m^*+\dots+e_1^*\otimes\dots\otimes\nabla e_m^*$. We may assume that $\tnabla e_i^*=0$. Then $\nabla e_i^*=(\nabla-\tnabla) e_i^*=A e_i^*$, leading to $|\nabla(\mbox{vol}_{\tg})|=O(r^{\lambda-3})$.
This shows that $F(\alpha)\in C^1_{(\mu-2,\lambda-2)}(\mathcal{C})$.
Further
calculations of the same type apply to the higher derivatives, showing that
$F(\alpha)\in C^\infty_{(\mu-2,\lambda-2)}(\mathcal{C})$. It is clear that $F$
is smooth.

The third statement can be proved as in Proposition \ref{prop:cptnonlinear}. 
\end{proof}


\subsection{Third case: CS/AC special Lagrangians in $\C^m$}\label{ss:accsslsetup}

Let $L$ be a AC, CS or CS/AC SL in $\C^m$ or a CS SL in $M$. The moduli space of
SL deformations of $L$ with fixed singularities coincides locally with the zero
set of a map $F$ defined as in Equation \ref{eq:defF}. The methods and results
of Sections \ref{ss:accslagdefs} and \ref{ss:coneslsetup} then lead to a good
understanding of the properties of $F$, analogous to those described in
Propositions \ref{prop:cptnonlinear} and \ref{prop:conenonlinear}. For example,
assume $L$ is a CS SL in $M$ with rate
$\boldsymbol{\mu}$. Let $\mathcal{D}_{L}$ denote the
space of closed 1-forms in
$C^\infty_{\boldsymbol{\mu}-1}(\Lambda^1)$ whose graph
lies in $\mathcal{U}$. Choose $\alpha\in \mathcal{D}_{L}$. Then one can 
check that $F(\alpha)\in
C^\infty_{\boldsymbol{\mu}-2}(L)$. The calculation is similar to the one already used in the proof of Proposition \ref{prop:conenonlinear}. In particular it uses (i) the fact that the asymptotic cones $\mathcal{C}_i$ are SL, (ii) the fact that the discrepancy between the forms $\Omega$ and $\tilde{\Omega}$ is of the order $O(r)<O(r^{\mu_i-2})$.

We now want to understand how to parametrize the SL deformations of $L$ whose
singularities are allowed to move in the ambient space as in Section
\ref{ss:movingsings}. For example, assume $L$ is a CS SL submanifold in $M$. The
constructions of Section \ref{ss:movingsings} must then be modified as follows.
This time we set
\begin{equation}
\tilde{P}:=\{(p,\upsilon):p\in M,\ \upsilon:\C^m\rightarrow T_pM \mbox{ such
that }\upsilon^*\omega=\tilde{\omega},\ \upsilon^*\Omega=\tilde{\Omega}\},
\end{equation}
so that $\tilde{P}$ is a $\sunitary m$-principal fibre bundle over $M$ of
dimension $m^2+2m-1$. For each end, the cone $\cone_i$ will now have symmetry
group $G_i\subset \sunitary m$. As in Section \ref{ss:movingsings}, let
$\tilde{\E}_i$ denote a smooth submanifold of $\tilde{P}$ transverse to the
orbits of $G_i$. It has dimension $m^2+2m-1-\mbox{dim}(G_i)$. Set
$\tilde{\E}:=\tilde{\E}_1\times\dots\times\tilde{\E}_s$. We then define CS
Lagrangian submanifolds $L_{\tilde{e}}$ and embeddings $\Phi_{L}^{\tilde{e}}$ with
the same properties as before.

Now let $\mathcal{D}_{L}$ denote the space of closed 1-forms in
$C^\infty_{\boldsymbol{\mu}-1}(\Lambda^1)$ whose graph lies in $\mathcal{U}$. We
define a map 
\begin{equation}\label{eq:csdefF}
F:\tilde{\E}\times\mathcal{D}_{L}\rightarrow C^\infty_{\boldsymbol{\mu}-2} (L),\
\ (\tilde{e},\alpha)\mapsto
\star(\alpha^*(\Phi_{L}^{\tilde{e}*}\Imag\,\Omega)).
\end{equation}
\begin{prop} \label{prop:csnonlinear}
Let $L$ be a CS SL in $M$. Then the map $F$ has the following properties:
\begin{enumerate}
\item The set $F^{-1}(0)$ parametrizes the space of all SL deformations of $L$
which are $C^1$-close to $L$ away from the singularities and are asymptotic to
$\mathcal{C}_i$ with rate $\mu_i$ for some choice of
$(\tilde{p}_i,\tilde{\upsilon}_i)$ near $(p_i,\upsilon_i)$.
\item $F$ is a (locally) well-defined smooth map between Fr\'echet spaces. In
particular, for each $\alpha\in\mathcal{D}_{L}$, $F(\alpha)\in
C^\infty_{\boldsymbol{\mu}-2}(L)$. Furthermore, $\int_L F(\alpha)\,vol_g=0$.
\item There exists an injective linear map $\chi:T_e\tilde{\E}\rightarrow
C^\infty_{\boldsymbol{0}}(L)$ such that (i) $\chi(y)\equiv 0$ away from the
singularities and (ii) the linearized map $dF[0]:T_e\tilde{\E}\oplus
C^\infty_{\boldsymbol{\mu}-1}(\Lambda^1)\rightarrow
C^\infty_{\boldsymbol{\mu}-2}(L)$ satisfies
\begin{equation}\label{eq:cslinearization}
dF[0](y,\alpha)=\Delta_g\,\chi(y)+d^*\alpha.
\end{equation}
\end{enumerate}
\end{prop}
\begin{proof}
The first statement should be interpreted as explained in the proof of
Proposition \ref{prop:cptnonlinear}. The proof follows from the definitions of
$\tilde{\E}$ and $F$ and from the results of Section \ref{ss:movingsings}. The
second statement can be proved as in Propositions \ref{prop:cptnonlinear} and
\ref{prop:conenonlinear}.

Regarding the third statement, the linearization of $F$ with respect to
directions in $C^\infty_{\boldsymbol{\mu}-1}(\Lambda^1)$ can be computed as in
Proposition \ref{prop:cptnonlinear}. Now choose $y\in T_e\tilde{\E}$
corresponding to a curve $\tilde{e}_s\in\tilde{\E}$ such that $\tilde{e}_0=e$. Up to
identifying $\mathcal{U}$ with $M$ via $\Phi_{L}^e$, $\Phi_{L}^{\tilde{e}_s}$
defines a 1-parameter curve of symplectomorphisms $\phi_s$ of $M$ such that
$d/ds(\phi_s)_{|s=0}=v$, for some vector field $v$ on $M$. Thus, as in
Proposition \ref{prop:cptnonlinear},
\begin{eqnarray*}
dF[0](y)\,vol_g &=&
d/ds(F(\tilde{e}_s,0)\,vol_g)_{|s=0}=d/ds((\phi_s)^*\Imag\,\Omega)_{|L;s=0}\\
&=& (\mathcal{L}_v\Imag\,\Omega)_{L}=(di_v\Imag\Omega)_{|L}\\
&=&-d\star\alpha,
\end{eqnarray*}
where $\alpha:=\omega(v,\cdot)_{|L}$ is a closed 1-form on $L$. Notice that, by
definition, $\phi_s\equiv Id$ away from the singularities of $L$, so
$\alpha\equiv 0$ there. Thus, by the Poincar\'e Lemma (cf. \textit{e.g.} Lemma
\ref{lemma:formclosed}), $\alpha$ must be exact on $L$, \textit{i.e.}
$\alpha=d\chi$ for some function $\chi:L\rightarrow \R$. We can define $\chi$
uniquely by imposing that $\chi\equiv 0$ away from the singularities of $L$. The
function $\chi$ depends linearly on $y$, and we can write
$dF[0](y,0)=\Delta_g\,\chi(y)$, as claimed. Furthermore, if $\chi(y)=0$ then
$\alpha=0$ and $v=0$. Since $\tilde{\E}$ is defined so as to parametrize
geometrically distinct immersions, this implies $y=0$.

Roughly speaking, near each singularity and up to the appropriate
identifications, $\tilde{e}_s$ should be thought of as a 1-parameter curve in the
group $\sunitary m\ltimes\C^m$ acting on $\C^m$. This action admits a
\textit{moment map} $\mu:\C^m\rightarrow (Lie(\sunitary m\ltimes\C^m))^*$.
Recall that this means that $\mu$ is equivariant and that, for all $w\in
Lie(\sunitary m\ltimes\C^m)$, the corresponding function
$\mu_w:\C^m\rightarrow\R$ satisfies $d\mu_w=i_w\tilde{\omega}$, \textit{i.e.}
$w$ is a Hamiltonian vector field with Hamiltonian function $\mu_w$. The moment
map can be written explicitly, cf. \textit{e.g.} \cite{haskinspacini} Section
2.6, showing that each $\mu_w$ is at most a quadratic polynomial on $\C^m$.
Notice, for future reference, that for any SL $L\subset\C^m$ the
calculations in the proof of Proposition \ref{prop:cptnonlinear} show that
\begin{equation*}
\Delta_g(\mu_{w|L})=d^*(d\mu_{w|L})=-\star
d\star(i_w\tilde{\omega}_{|L})=\star (d
i_w\Imag\,\tilde{\Omega})_{|L}=\star(\mathcal{L}_w\Imag\,\tilde{\Omega
})_{|L}=0,
\end{equation*}
\textit{i.e.} each $\mu_w$ restricts to a harmonic function on $L$. 

In this set-up our vector field $v$ is (locally) an element of $Lie(\sunitary
m\ltimes\C^m)$ and $\chi(y)=\mu_v$. Thus $\chi(y)$ is bounded as $r\rightarrow
0$. This implies that $\chi(y)\in C^0_{\boldsymbol{0}}(L)$. Further calculations
show that $\chi(y)\in C^\infty_{\boldsymbol{0}}(L)$, as claimed.
\end{proof}


Now let $L$ be a CS/AC SL in $\C^m$. Define $\tilde{P}$, $\tilde{\E}$,
\textit{etc.} analogously to the above (cf. Section \ref{ss:movingsings} for the
necessary modifications for the ambient space $\C^m$). Let $\mathcal{D}_{L}$
denote the space of closed 1-forms in
$C^\infty_{(\boldsymbol{\mu}-1,\boldsymbol{\lambda}-1)}(\Lambda^1)$ whose graph
lies in $\mathcal{U}$. Define $F$ as in Equation \ref{eq:csdefF}. 

\begin{prop}\label{prop:accsnonlinear}
Let $L$ be a CS/AC SL in $\C^m$. Then the map $F$ has the following properties:
\begin{enumerate}
\item The set $F^{-1}(0)$ parametrizes the space of all SL deformations of $L$
which are $C^1$-close to $L$ away from the singularities and are asymptotic to
$\mathcal{C}_i$ with rate $(\boldsymbol{\mu},\boldsymbol{\lambda})$ for some
choice of $(\tilde{p}_i,\tilde{\upsilon}_i)$ near $(p_i,\upsilon_i)$.
\item $F$ is a (locally) well-defined smooth map between Fr\'echet spaces. In
particular, for each $\alpha\in\mathcal{D}_{L}$, $F(\alpha)\in
C^\infty_{(\boldsymbol{\mu}-2,\boldsymbol{\lambda}-2)}(L)$.
\item There exists an injective linear map $\chi:T_e\tilde{\E}\rightarrow
C^\infty_{\boldsymbol{0}}(L)$ such that (i) $\chi(y)\equiv 0$ away from the
singularities and (ii) the linearized map $dF[0]:T_e\tilde{\E}\oplus
C^\infty_{(\boldsymbol{\mu}-1,\boldsymbol{\lambda}-1)}(\Lambda^1)\rightarrow
C^\infty_{(\boldsymbol{\mu}-2,\boldsymbol{\lambda}-2)}(L)$ satisfies
\begin{equation}
dF[0](y,\alpha)=\Delta_g\,\chi(y)+d^*\alpha.
\end{equation}
\end{enumerate}
\end{prop}
\begin{proof}
 The proof is similar to that of Proposition \ref{prop:csnonlinear}, but notice that in this case we obtain stronger control over the properties of the function $\chi(y)=\mu_v$, near the singularities: the fact that $\mu_v$ restricts to a harmonic function on $L$ implies that $\Delta_g\,\chi(y)$ vanishes in a neighbourhood of each singularity. 
\end{proof}

If the spaces $C^\infty(L)$, $C^\infty_{\boldsymbol{\beta}}(L)$ were Banach
spaces and the relevant maps were Fredholm, we could now apply the Implicit
Function Theorem to conclude that the sets $F^{-1}(0)$, and thus
$\mathcal{M}_L$, are smooth. As however they are actually only Fr\'echet spaces,
it is instead necessary to first take the Sobolev space completions of these
spaces, then study the Fredholm properties of the linearized maps. We do this in
Section \ref{s:moduli}. This will require some results concerning the Laplace
operator on conifolds, summarized in Section \ref{s:reviewlaplace}.


\section{Review of the Laplace operator on conifolds}\label{s:reviewlaplace}

We summarize here some analytic results concerning the Laplace operator on
conifolds, referring to \cite{pacini:weighted} for further details and
references.

\begin{definition}\label{def:exceptionalweights}
Let $(\Sigma,g')$ be a compact Riemannian manifold. Consider the cone
$C:=\Sigma\times (0,\infty)$ endowed with the conical metric
$\tilde{g}:=dr^2+r^2g'$. Let $\Delta_{\tilde{g}}$ denote the corresponding
Laplace operator acting on functions.

For each component $(\Sigma_j,g_j')$ of $(\Sigma,g')$ and each $\gamma\in\R$,
consider the space of homogeneous harmonic functions 
\begin{equation}\label{eq:ac_harmonicter}
V^j_{\gamma}:=\{r^\gamma\sigma(\theta):
\Delta_{\tilde{g}}(r^{\gamma}\sigma)=0\}.
\end{equation}
Set $m^j(\gamma):=\mbox{dim}(V^j_\gamma)$. One can show that $m^j_{\gamma}>0$
iff $\gamma$ satisfies the equation
\begin{equation}\label{eq:exceptionalforlaplacian}
\gamma=\frac{(2-m)\pm\sqrt{(2-m)^2+4e_n^j}}{2},
\end{equation}
for some eigenvalue $e_n^j$ of $\Delta_{g_j'}$ on $\Sigma_j$. Given any weight
$\boldsymbol{\gamma}\in \R^e$, we now set $m(\boldsymbol{\gamma}):=\sum_{j=1}^e
m^j(\gamma_j)$. Let $\mathcal{D}\subseteq\R^e$ denote the set of weights
$\boldsymbol{\gamma}$ for which $m(\boldsymbol{\gamma})>0$. We call these the
\textit{exceptional weights} of $\Delta_{\tg}$.
\end{definition}

Let $(L,g)$ be a conifold. Assume $(L,g)$ is asymptotic to a cone $(C,\tg)$ in
the sense of Definition \ref{def:conifold}. Roughly speaking, the fact that $g$
is asymptotic to $\tilde{g}$ in the sense of Definition \ref{def:metrics_ends}
implies that the Laplace operator $\Delta_g$ is ``asymptotic" to $\Delta_{\tg}$.
Applying Definition \ref{def:exceptionalweights} to $C$ defines weights
$\mathcal{D}\subseteq\R^e$: we call these the \textit{exceptional weights} of
$\Delta_g$. This terminology is due to the following result.

\begin{theorem}\label{thm:laplaceresults}
Let $(L,g)$ be a conifold with $e$ ends. Let $\mathcal{D}$ denote
the exceptional weights of $\Delta_g$. Then $\mathcal{D}$ is a discrete subset of $\R^e$ and the Laplace operator 
\begin{equation*}
\Delta_g:W^p_{k,\boldsymbol{\beta}}(L)\rightarrow
W^p_{k-2,\boldsymbol{\beta}-2}(L)
\end{equation*}
is Fredholm iff $\boldsymbol{\beta}\notin \mathcal{D}$.
\end{theorem}

The above theorem, coupled with the ``change of index formula", leads to the
following conclusion, cf. \cite{pacini:weighted}.

\begin{corollary}\label{cor:laplaceresults}
Let $(L,g)$ be a compact Riemannian manifold. Consider the map
$\Delta_g:W^p_k(L)\rightarrow W^p_{k-2}(L)$. Then
\begin{equation*}
\mbox{Im}(\Delta_g)=\{u\in W^p_{k-2}(L):\int_L u\,\mbox{vol}_g=0\}, \ \
\mbox{Ker}(\Delta_g)=\R.
\end{equation*}
Let $(L,g)$ be an AC manifold. Consider the map
$\Delta_g:W^p_{k,\boldsymbol{\lambda}}(L)\rightarrow
W^p_{k-2,\boldsymbol{\lambda}-2}(L)$.
If $\boldsymbol{\lambda}>2-m$ is non-exceptional then this map is surjective. If
$\boldsymbol{\lambda}<0$ then this map is injective, so for
$\boldsymbol{\lambda}\in (2-m,0)$ it is an isomorphism. 

Let $(L,g)$ be a CS manifold with $e$ ends. Consider the map
$\Delta_g:W^p_{k,\boldsymbol{\mu}}(L)\rightarrow
W^p_{k-2,\boldsymbol{\mu}-2}(L)$. If $\boldsymbol{\mu}\in (2-m,0)$ then
\begin{equation*}
\mbox{Im}(\Delta_g)=\{u\in W^p_{k-2,\boldsymbol{\mu}-2}(L): \int_L u\,vol_g=0\},
\ \ Ker(\Delta_g)=\R.
\end{equation*} 
If $\boldsymbol{\mu}>0$ is non-exceptional then this map is injective and 
\begin{equation*}
\mbox{dim(Coker($\Delta_g$))}=e+\sum_{0<\boldsymbol{\gamma}<\boldsymbol{\mu}}
m(\boldsymbol{\gamma}),
\end{equation*}
where $m(\boldsymbol{\gamma})$ is as in Definition \ref{def:exceptionalweights}.

Let $(L,g)$ be a CS/AC manifold with $s$ CS ends and $l$ AC ends. Consider the
map
\begin{equation*}
\Delta_g:W^p_{k,(\boldsymbol{\mu},\boldsymbol{\lambda})}(L)\rightarrow
W^p_{k-2,(\boldsymbol{\mu}-2,\boldsymbol{\lambda}-2)}(L).
\end{equation*}
If $(\boldsymbol{\mu},\boldsymbol{\lambda})\in (2-m,0)$ then this map is an
isomorphism. If $\boldsymbol{\mu}>0$ and $\boldsymbol{\lambda}<0$ are
non-exceptional then this map is injective and 
\begin{equation*}
\mbox{dim(Coker($\Delta_g$))}=s+\sum_{0<\boldsymbol{\gamma}<\boldsymbol{\mu}}
m(\boldsymbol{\gamma}),
\end{equation*}
where $m(\boldsymbol{\gamma})$ is as in Definition \ref{def:exceptionalweights}.
Notice in particular that this dimension depends only on the harmonic functions
on the CS cones.
\end{corollary}


\section{Moduli spaces of special Lagrangian conifolds}\label{s:moduli}

Recall the statement of the Implicit Function Theorem.
\begin{theorem} \label{thm:IFT}
Let $F:E_1\rightarrow E_2$ be a smooth map between Banach spaces such that
$F(0)=0$. Assume $P:=dF[0]$ is surjective and  $\mbox{Ker}(P)$ admits a closed
complement $Z$, \textit{i.e.} $E_1=\mbox{Ker}(P)\oplus Z$. Then there exists a
smooth map $\Phi:\mbox{Ker}(P)\rightarrow Z$ such that $F^{-1}(0)$ coincides
locally with the graph $\Gamma(\Phi)$ of $\Phi$. In particular, $F^{-1}(0)$ is
(locally) a smooth Banach submanifold of $E_1$.
\end{theorem}

The following result is straight-forward. 

\begin{prop}\label{prop:fredholmreducestofinitedim}
Let $F:E_1\rightarrow E_2$ be a smooth map between Banach spaces such that
$F(0)=0$. Assume $P:=dF[0]$ is Fredholm. Set $\mathcal{I}:=\mbox{Ker}(P)$ and
choose $Z$ such that $E_1=\mathcal{I}\oplus Z$. Let $\mathcal{O}$ denote a
finite-dimensional subspace of $E_2$ such that $E_2=\mathcal{O}\oplus
\mbox{Im}(P)$. Define
\begin{equation*}
G:\mathcal{O}\oplus E_1\rightarrow E_2,\ \ (\gamma,e)\mapsto \gamma+F(e).
\end{equation*}
Identify $E_1$ with $(0,E_1)\subset \mathcal{O}\oplus E_1$. Then:
\begin{enumerate}
\item The map $dG[0]=Id\oplus P$ is surjective and
$\mbox{Ker}(dG[0])=\mbox{Ker}(P)$. Thus, by the Implicit Function Theorem, there
exist $\Phi:\mathcal{I}\rightarrow\mathcal{O}\oplus Z$ such that
$G^{-1}(0)=\Gamma(\Phi)$.
\item $F^{-1}(0)=\{(i,\Phi(i)):\Phi(i)\in
Z\}=\{(i,\Phi(i)):\pi_{\mathcal{O}}\circ\Phi(i)=0\}$,
where $\pi_{\mathcal{O}}:\mathcal{O}\oplus Z\rightarrow\mathcal{O}$ is the
standard projection.
\item Let $\pi_{\mathcal{I}}:\mathcal{I}\oplus Z\rightarrow \mathcal{I}$ denote
the standard projection. Then $\pi_{\mathcal{I}}$ is a continuous open map so it
restricts to a homeomorphism
\begin{equation*}
\pi_{\mathcal{I}}:F^{-1}(0)\rightarrow (\pi_{\mathcal{O}}\circ\Phi)^{-1}(0)
\end{equation*}
between $F^{-1}(0)$ and the zero set of the smooth map
$\pi_{\mathcal{O}}\circ\Phi:\mathcal{I}\rightarrow\mathcal{O}$, which is defined
between finite-dimensional spaces.
\end{enumerate}
\end{prop}

We now have all the ingredients necessary to prove various smoothness results
for SL moduli spaces. In all cases we follow the same steps. Section
\ref{s:setup} described each moduli space as the zero set of a map $F$. The
first step is to use regularity to show that one can equivalently study the zero
set of a map $\tilde{F}$. The domain of $\tilde{F}$ is of the form $K\times
W^p_{k,(\boldsymbol{\mu},\boldsymbol{\lambda})}(L)$ where $K$ is a
finite-dimensional vector space defined in terms of spaces introduced in
Sections \ref{ss:closedforms} and \ref{s:setup}. Roughly speaking, this
corresponds to separating the obvious Hamiltonian deformations of $L$ from a
finite-dimensional space of other Lagrangian deformations. The geometric
description of the latter depends on the case in question. The differential
$d\tilde{F}[0]$ is then a finite-dimensional perturbation of the Laplace
operator $\Delta_g$ acting on functions. The second step is to analyze this
linearized    operator, showing that under appropriate conditions it is
surjective. The third step is to identify the kernel of $d\tilde{F}[0]$, at
least up to projections. One can then apply the Implicit Function Theorem and
conclude.

\subsubsection*{Smooth compact special Lagrangians} The following result was first proved by McLean \cite{mclean}.

\begin{theorem} \label{thm:mclean}
Let $L$ be a smooth compact SL submanifold of a CY manifold $M$. Let
$\mathcal{M}_L$ denote the moduli space of SL deformations of $L$. Then
$\mathcal{M}_L$ is a smooth manifold of dimension $b^1(L)$. 
\end{theorem}
\begin{proof}
Choose $k\geq 3$ and $p>m$. Consider the space $\mbox{Ker}(d)$ of closed 1-forms
in $W^p_{k-1}(\Lambda^1)$. Let $\mathcal{D}_L$ denote the forms
$\alpha\in\mbox{Ker}(d)$ whose graph $\Gamma(\alpha)$ lies in $\mathcal{U}$.
Notice that $\Gamma(\alpha)$ is a well-defined $C^1$ Lagrangian submanifold in
$\mathcal{U}$ by the standard Sobolev embedding
$W^p_{k-1}(\Lambda^1)\hookrightarrow C^1(\Lambda^1)$. For the same reason,
$\mathcal{D}_L$ is an open neighbourhood of the origin in $\mbox{Ker}(d)$.
Consider the map 
\begin{equation}\label{eq:cptFsobolev}
F:\mathcal{D}_L\rightarrow \{u\in W^p_{k-2}(L):\int_L u\,vol_g=0\}, \ \
\alpha\mapsto \star(\pi_*((\Phi_L^*\Imag\,\Omega)_{|\Gamma(\alpha)})).
\end{equation}
Recall that $W^p_{k-2}(L)$ is closed under multiplication.
Together with the ideas of Proposition \ref{prop:cptnonlinear}, this shows that
$F$ is a (locally well-defined) smooth map between Banach spaces with
differential $dF[0](\alpha)=d^*\alpha$. Assume $\alpha\in F^{-1}(0)$. Then, by
composition with $\Phi_L$, $\alpha$ defines a $C^1$ SL submanifold in $M$.
Standard regularity results for minimal submanifolds then show that $\alpha$ is
smooth. Thus $\mathcal{M}_L$ is locally homeomorphic, via $\Phi_L$, to
$F^{-1}(0)$. 

Decomposition \ref{decomp:closedforms} shows that any $\alpha\in F^{-1}(0)$ is
of the form $\alpha=\beta+df$ for some unique $\beta\in H$ and some $f\in
C^\infty(L)$, defined up to a constant. We can thus re-phrase the SL deformation
problem as follows. Define $\mathcal{\tD}_L$ as the space of pairs $(\beta,f)$
in $H\times W^p_k(L)$ such that $\alpha:=\beta+df\in \mathcal{D}_L$. Clearly
$\mathcal{\tD}_L$ is an open neighbourhood of the origin. Then $\mathcal{\tD}_L$
is the domain of the (locally defined) map between Banach spaces
\begin{equation}\label{eq:cptFsobolevbis}
\tilde{F}:H\times W^p_k(L)\rightarrow \{u\in W^p_{k-2}(L):\int_L u\,vol_g=0\},\
\ \tilde{F}(\beta,f):=F(\beta+df).
\end{equation} 
Clearly, $d\tilde{F}[0](\beta,f)=d^*\beta+\Delta_g f$. Let $\R$ denote the space
of constant functions in $W^p_k(L)$. Notice that both $\mathcal{\tD}_L$ and
$\tilde{F}$ are invariant under translations in $\R$. Assume
$\tilde{F}(\beta,f)=0$. With respect to $f$ this is a second-order elliptic
equation. Standard regularity results show that $f$ is smooth. This proves that
$\mathcal{M}_L$ is locally homeomorphic to the quotient space
$\tilde{F}^{-1}(0)/\R$. To conclude, it is sufficient to prove that
$\tilde{F}^{-1}(0)$ is smooth. According to Corollary \ref{cor:laplaceresults},
the map
\begin{equation}
\Delta_g:W^p_k(L)\rightarrow \{u\in W^p_{k-2}(L):\int_Lu\,vol_g=0\}
\end{equation}
is surjective. This implies that $d\tilde{F}[0]$ is surjective. Let $\beta_i$ be
a basis for $H$. For each $\beta_i$ the equation $d\tilde{F}[0](\beta_i,f)=0$
admits a solution $f_i$. Another solution is given by the pair $\beta=0$, $f=1$.
It is simple to check that these give a basis for the kernel of $d\tilde{F}[0]$.
Applying the Implicit Function Theorem we conclude that $\tilde{F}^{-1}(0)$ is
smooth of dimension $b^1(L)+1$, thus $\mathcal{M}_L$ is smooth of dimension
$b^1(L)$.
\end{proof}

\subsubsection*{AC special Lagrangians} The analogous result for AC SLs was originally proved independently by the
author \cite{pacini:defs} and by Marshall \cite{marshall}. We present here a
simplified proof, starting with the following weighted regularity result due to
Joyce, cf. \cite{joyce:I} Theorems 5.1 and 7.7.
\begin{lemma}\label{l:SLweightedregularity}
Let $\mathcal{C}$ be a SL cone in $\C^m$, endowed with the induced metric
$\tilde{g}$ and orientation. Define $\Phi_{\mathcal{C}}:\mathcal{U}\rightarrow
\C^m$ and the map $F$ as in Section \ref{ss:coneslsetup}. Fix any $\mu>2$ and
$\lambda<2$ with $\lambda\neq 0$. Assume given a closed 1-form $\alpha\in
C^1_{(\mu-1,\lambda-1)}(\mathcal{U})$ satisfying $F(\alpha)=0$. Analogously to
Decomposition \ref{decomp:closedforms_accs}, we can write $\alpha=\alpha'+dA$
where (i) $\alpha'$ is compactly-supported on the small end and
translation-invariant on the large end, and (ii) $A\in C^1_{(\mu,\lambda)}(L)$.
Then $\alpha'$ is smooth and $A\in C^\infty_{(\mu,\lambda)}(L)$, so $\alpha\in
C^\infty_{(\mu-1,\lambda-1)}(\mathcal{U})$. 
\end{lemma}
\begin{proof} Standard regularity results for minimal submanifolds show that
$\alpha\in C^1_{(\mu-1,\lambda-1)}(\mathcal{U})\cap C^\infty(\mathcal{U})$.
Using the same ideas as in the proof of Decomposition
\ref{decomp:closedforms_accs}, this suffices to prove that $\alpha'$ and $A$ are
smooth. It is thus enough to show that the higher derivatives of $A$ converge at
the correct rate as $r\rightarrow\infty$ and $r\rightarrow 0$. We sketch here a
proof for $r\rightarrow\infty$, referring to \cite{joyce:I} for details; the
other case is analogous.

In terms of $A$, \textit{i.e.} absorbing the $\alpha'$-terms into the operator,
the equation $F(\alpha)=0$ corresponds to an equation $\tilde{F}(A)=0$. Given
$r_0>0$ and $\epsilon<<1$, consider the equivalent equation 
\begin{equation}\label{eq:conformalSLeq}
r^2\tilde{F}(A)=0\ \ \mbox{restricted to }\Sigma\times(r_0-\epsilon
r_0,r_0+\epsilon r_0).
\end{equation}
As in Theorem \ref{thm:mclean} the linearization of $\tilde{F}$ is
$\Delta_{\tilde{g}}$. One can check that the change of coordinates $r=e^z$
transforms Equation \ref{eq:conformalSLeq} into an equation of the form 
\begin{equation}\label{eq:conformalSLeqbis}
\Delta_{\tilde{h}}(A)+\dots=0\ \ \mbox{restricted to
}\Sigma\times(r_0'-\epsilon', r_0'+\epsilon'), 
\end{equation}
where $\tilde{h}$ is the ``cylindrical metric"
$\tilde{h}:=r^{-2}\tilde{g}=dz^2+g'$. Up to a translation we can identify
$\Sigma\times(r_0'-\epsilon', r_0'+\epsilon')$ with the fixed, \textit{i.e.}
$r_0$-independent, domain $\Sigma\times(-\epsilon', \epsilon')$. One can show
that Equation \ref{eq:conformalSLeqbis} converges to the equation
$\Delta_{\tilde{h}}(A)=0$ on this domain in such a way that interior estimates
for the solutions are uniform as $r_0\rightarrow\infty$. In particular, in terms
of H\"{o}lder norms, there exists a constant $C=C(k,\beta)$ independent of $r_0$
such that
\begin{equation}\label{eq:SLregestimate}
\|A\|_{C^{k,\beta}}\leq C\cdot\|A\|_{C^1}
\end{equation}
on the domain $\Sigma\times(-\epsilon', \epsilon')$ and with respect to the
metric $\tilde{h}$. To be precise, as this is an ``interior'' estimate, the
domain on the left hand side is slightly smaller than the domain on the right
hand side.

Let us now write this estimate in terms of the coordinate $r$ and multiply both
sides by $r^{-\lambda}$. We can then check that
\begin{equation}\label{eq:SLregestimatebis}
\|A\|_{C^k_\lambda}\leq C\cdot\|A\|_{C^1_\lambda}
\end{equation}
on the domain $\Sigma\times(r_0-\epsilon r_0, r_0+\epsilon r_0)$ and with
respect to the metric $\tilde{g}$. As $r_0$ is arbitrary and
$\|A\|_{C^1_\lambda}$ is bounded on the large end, this shows that
$\|A\|_{C^k_\lambda}$ is bounded for all $k$ so $A\in C^\infty_\lambda$.
\end{proof}

\begin{theorem} \label{thm:pacini}
Let $L$ be an AC SL submanifold of $\C^m$ with rate
$\boldsymbol{\lambda}$. Let $\mathcal{M}_L$ denote the moduli space of SL
deformations of $L$ with rate $\boldsymbol{\lambda}$. Consider the operator
\begin{equation}\label{eq:aclaplacian}
\Delta_g:W^p_{k,\boldsymbol{\lambda}}(L)\rightarrow
W^p_{k-2,\boldsymbol{\lambda}-2}(L).
\end{equation}
\begin{enumerate}
\item If $\boldsymbol{\lambda}\in (0,2)$ is a non-exceptional weight for
$\Delta_g$ then $\mathcal{M}_L$ is a smooth manifold of dimension
$b^1(L)+\mbox{dim(Ker$(\Delta_g)$)}-1$. 
\item If $\boldsymbol{\lambda}\in (2-m,0)$ then $\mathcal{M}_L$ is a smooth
manifold of dimension $b^1_c(L)$.
\end{enumerate}
\end{theorem}
\begin{proof}
As in Theorem \ref{thm:mclean}, choose $k\geq 3$ and $p>m$ so that
$W^p_{k-1,\boldsymbol{\lambda}-1}(\Lambda^1)\subset
C^1_{\boldsymbol{\lambda}-1}(\Lambda^1)$. Let $\mathcal{D}_L$ denote the space
of closed 1-forms in $W^p_{k-1,{\boldsymbol{\lambda}-1}}(\Lambda^1)$ whose graph
$\Gamma(\alpha)$ lies in $\mathcal{U}$. Consider the map 
\begin{equation}\label{eq:acFsobolev}
F:\mathcal{D}_L\rightarrow W^p_{k-2,\boldsymbol{\lambda}-2}(L), \ \
\alpha\mapsto \star(\pi_*((\Phi_L^*\Imag\,\tilde{\Omega})_{|\Gamma(\alpha)})).
\end{equation}
Assume $\boldsymbol{\lambda}<2$. In this case Theorem
\ref{thm:embedding} shows that
$W^p_{k-2,\boldsymbol{\lambda}-2}(L)$ is closed under
multiplication. Together with the ideas of Proposition
\ref{prop:cptnonlinear}, this shows that $F$ is a (locally well-defined) smooth
map between Banach spaces with differential $dF[0](\alpha)=d^*\alpha$. Assume
$F(\alpha)=0$. Theorem \ref{thm:embedding} and Lemma
\ref{l:SLweightedregularity} then show that $\alpha\in
C^\infty_{\boldsymbol{\lambda}-1}(\Lambda^1)$ so $F^{-1}(0)$ is locally
homeomorphic, via $\Phi_L$, to $\mathcal{M}_L$.

Now assume $\boldsymbol{\lambda}\in (0,2)$. Decomposition
\ref{decomp:closedforms_growth} shows that any $\alpha\in F^{-1}(0)$ is of the
form $\alpha=\beta+df$, for some $\beta\in H$ and some $f\in
C^\infty_{\boldsymbol{\lambda}}(L)$. Define $\mathcal{\tD}_L$ as the space of
pairs $(\beta,f)$ in $H\times W^p_{k,\boldsymbol{\lambda}}(L)$ such that
$\alpha:=\beta+df\in \mathcal{D}_L$. Clearly $\mathcal{\tD}_L$ is an open
neighbourhood of the origin. Then $\mathcal{\tD}_L$ is the domain of the
(locally defined) smooth map between Banach spaces
\begin{equation}\label{eq:acFsobolevbis}
\tilde{F}:H\times W^p_{k,\boldsymbol{\lambda}}(L)\rightarrow
W^p_{k-2,\boldsymbol{\lambda}-2}(L),\ \ \tilde{F}(\beta,f):=F(\beta+df)
\end{equation} 
with $d\tilde{F}[0](\beta,f)=d^*\beta+\Delta_g f$ and invariant under
translations in $\R$. Assume $\tilde{F}(\beta,f)=0$. Theorem \ref{thm:embedding}
and Lemma \ref{l:SLweightedregularity} then show that $f\in
C^\infty_{\boldsymbol{\lambda}}(L)$. This proves that $\mathcal{M}_L$ is locally
homeomorphic, via $\Phi_L$, to the quotient space $\tilde{F}^{-1}(0)/\R$. To
conclude, it is thus sufficient to prove that $\tilde{F}^{-1}(0)$ is smooth. For
this we need to further assume that $\boldsymbol{\lambda}$ is non-exceptional.
Then Corollary \ref{cor:laplaceresults} shows that the map of Equation
\ref{eq:aclaplacian} is surjective, so $d\tilde{F}[0]$ is surjective. Let
$\beta_i$ be a basis for $H$. For each $\beta_i$ the equation
$d\tilde{F}[0](\beta_i,f)=0$ admits a solution $f_i$. More solutions are given
by the pairs $\beta=0$, $f\in Ker(\Delta_g)$. It is simple to check that these
give a basis for the kernel of $d\tilde{F}[0]$. Applying the Implicit Function
Theorem we conclude that $\tilde{F}^{-1}(0)$ is smooth of dimension
$\mbox{dim}(H\oplus\mbox{Ker}(\Delta_g))$. Thus $\mathcal{M}_L$ is smooth and
has the claimed dimension.

Now assume $\boldsymbol{\lambda}\in (2-m,0)$. In this case Decomposition
\ref{decomp:closedforms_decay} shows that any $\alpha\in F^{-1}(0)$ is of the
form $\alpha=\beta+dv+df$, for some $\beta\in \widetilde{H}_\infty$, $dv\in
d(E_\infty)$ and $df\in d(C^\infty_{\boldsymbol{\lambda}}(L))$. We can use
regularity as before to prove that $\mathcal{M}_L$ is locally homeomorphic to
the quotient space $\tilde{F}^{-1}(0)/\R$, for the (locally defined) map 
\begin{equation}\label{eq:acFsobolevter}
\tilde{F}:\widetilde{H}_\infty\times E_\infty\times
W^p_{k,\boldsymbol{\lambda}}(L)\rightarrow W^p_{k-2,\boldsymbol{\lambda}-2}(L),\
\ \tilde{F}(\beta,v,f)=F(\beta+dv+df).
\end{equation}
Notice that this time the constant functions $\R$ are contained in $E_\infty$.
We conclude as before that $\tilde{F}^{-1}(0)$ is smooth, this time of dimension
$\mbox{dim}(\widetilde{H}_\infty\oplus E_\infty)$. Remark
\ref{rem:compactsupport} then shows that $\mathcal{M}_L$ is smooth of dimension
$b^1_c(L)$.
\end{proof}

\subsubsection*{CS special Lagrangians} Now assume that $L$ is CS SL with singularities modelled on cones
$\mathcal{C}_i$. It turns out that smoothness of $\mathcal{M}_L$ then requires
an additional ``stability" assumption on $\mathcal{C}_i$. Roughly speaking, it
is required that the cones $\mathcal{C}_i$ admit no additional harmonic
functions with prescribed growth, beyond those which necessarily exist for
geometric reasons.

\begin{definition}\label{def:stability}
Let $\mathcal{C}$ be a SL cone in $\C^m$. Let $(\Sigma,g')$ denote the link of
$\mathcal{C}$ with the induced metric. Assume $\mathcal{C}$ has a unique
singularity at the origin; equivalently, assume that $\Sigma$ is smooth and that
it is not a sphere $\Sph^{m-1}\subset \Sph^{2m-1}$. Recall from the proof of
Proposition \ref{prop:csnonlinear} that the standard action of $\sunitary
m\ltimes\C^m$ on $\C^m$ admits a moment map $\mu$ and that the components of
$\mu$ restrict to harmonic functions on $\mathcal{C}$. Let $G$ denote the
subgroup of $\sunitary m$ which preserves $\mathcal{C}$. Then $\mu$ defines on
$\mathcal{C}$ $2m$ linearly independent harmonic functions of linear growth; in
the notation of Definition \ref{def:exceptionalweights} these functions are
contained in the space $V_\gamma$ with $\gamma=1$. The moment map also defines
on $\mathcal{C}$ $m^2-1-\mbox{dim}(G)$ linearly independent harmonic functions
of quadratic growth: these belong to the space $V_\gamma$ with $\gamma=2$.
Constant functions define a third space of homogeneous harmonic functions on
$\mathcal{C}$, \textit{i.e.} elements in $V_\gamma$ with $\gamma=0$. In
particular, these three values of $\gamma$ are always exceptional values for the
operator $\Delta_{\tilde{g}}$ on any SL cone, in the sense of Definition
\ref{def:exceptionalweights}.

We say that $\mathcal{C}$ is \textit{stable} if these are the only functions in
$V_\gamma$ for $\gamma=0,1,2$ and if there are no other exceptional values
$\gamma$ in the interval $[0,2]$. More generally, let $L$ be a CS or CS/AC SL
submanifold. We say that a singularity $x_i$ of $L$ is \textit{stable} if the
corresponding cone $\mathcal{C}_i$ is stable.
\end{definition}

The following result is due to Joyce \cite{joyce:II}.
\begin{theorem} \label{thm:joyce}
Let $L$ be a CS SL submanifold of $M$ with $s$ singularities and rate
$\boldsymbol{\mu}$. Let $\mathcal{M}_L$ denote the moduli space of SL
deformations of $L$ with moving singularities and rate $\boldsymbol{\mu}$.
Assume $\boldsymbol{\mu}$ is non-exceptional for the map
\begin{equation}\label{eq:cslaplacian}
\Delta_g:W^p_{k,\boldsymbol{\mu}}(L)\rightarrow \{u\in
W^p_{k-2,\boldsymbol{\mu}-2}(L): \int_L u\,vol_g=0\}.
\end{equation}
Then $\mathcal{M}_L$ is locally homeomorphic to the zero set of a smooth map
$\Phi: \mathcal{I}\rightarrow\mathcal{O}$
defined (locally) between finite-dimensional vector spaces. If
$\boldsymbol{\mu}=2+\epsilon$ and all singularities are stable then
$\mathcal{O}=\{0\}$ and $\mathcal{M}_L$ is smooth of dimension
$\mbox{dim}(\mathcal{I})=b^1_c(L)-s+1$.
\end{theorem}
\begin{proof}
Start with a map $F$ defined as in Section \ref{ss:accsslsetup} on
$\tilde{\E}\times W^p_{k-1,\boldsymbol{\mu}-1}(\Lambda^1)$. As in Theorem
\ref{thm:pacini}, regularity and Decomposition \ref{decomp:closedforms_decay}
show that $\mathcal{M}_L$ is locally homeomorphic to $\tilde{F}^{-1}(0)/\R$,
where $\tilde{F}$ is the (locally-defined) map
\begin{eqnarray*}
\tilde{F}:\tilde{\E}\times \widetilde{H}_0\times E_0\times
W^p_{k,\boldsymbol{\mu}}(L) &\rightarrow& \{u\in
W^p_{k-2,\boldsymbol{\mu}-2}(L):\int_Lu\,vol_g=0\}\\
(\tilde{e},\beta,v,f) &\mapsto& F(\tilde{e},\beta+dv+df),
\end{eqnarray*}
invariant under translations in $\R\subset E_0$. As in Proposition
\ref{prop:csnonlinear},
$d\tilde{F}[0](y,\beta,v,f)=d^*\beta+\Delta_g(\chi(y)+v+f)$. Now consider the
restricted map
\begin{equation}\label{eq:csrestrictedlin}
d\tilde{F}[0]:T_e\tilde{\E}\oplus E_0\oplus
W^p_{k,\boldsymbol{\mu}}(L)\rightarrow \{u\in
W^p_{k-2,\boldsymbol{\mu}-2}(L):\int_Lu\,vol_g=0\}.
\end{equation}
We claim that the kernel of this map is given by the constant functions $\R$. To
prove this, assume $d\tilde{F}[0](\chi(y)+v+f)=0$. Since $\chi(y)+v+f\in
W^p_{k,-\boldsymbol{\epsilon}}(L)$, Corollary \ref{cor:laplaceresults} shows
that $\chi(y)+v+f$ is constant, \textit{i.e.} $d(\chi(y)+v+f)=0$. In other words
the infinitesimal Lagrangian deformation of $L$ defined by $(y,v,f)$ is trivial,
so in particular $y=0$. This implies $\chi(y)=0$ and it is simple to conclude
that $f=0$ and $v\in\R$. 

Let $\mathcal{O}$ denote the cokernel of the map of Equation
\ref{eq:csrestrictedlin}. More precisely, we define it to be a
finite-dimensional space of $W^p_{k-2,\mu-2}(L)$ such that 
\begin{equation}
\mathcal{O}\oplus d\tilde{F}[0]\left(T_e\tilde{\E}\oplus E_0\oplus
W^p_{k,\boldsymbol{\mu}}\right)=\{u\in
W^p_{k-2,\boldsymbol{\mu}-2}(L):\int_Lu\,vol_g=0\}.
\end{equation}
Consider the map
\begin{eqnarray*}
G:\mathcal{O}\times \tilde{\E}\times \widetilde{H}_0\times E_0\times
W^p_{k,\boldsymbol{\mu}}(L) &\rightarrow& \{u\in
W^p_{k-2,\boldsymbol{\mu}-2}(L):\int_Lu \,vol_g=0\}\\
(\gamma,\tilde{e},\beta,v,f) &\mapsto& \gamma+\tilde{F}(\tilde{e},\beta,v,f).
\end{eqnarray*}
Again, $G$ is invariant under translations in $\R$. By construction, the
restriction of $dG[0]$ to the space $\mathcal{O}\oplus T_e\tilde{\E}\oplus
E_0\oplus W^p_{k,\boldsymbol{\mu}}$ is surjective with kernel $\R$. We now have
the following information about the map $G$. Firstly,
$\mbox{Ker}(dG[0])=V\oplus\R$, where $V$ is some vector space projecting
isomorphically onto $\widetilde{H}_0$. Secondly, by the Implicit Function
Theorem, the set $G^{-1}(0)$ is smooth and can be locally written as the graph
of a smooth map $\Phi$ defined on the kernel of $dG[0]$, thus on
$\widetilde{H}_0\oplus\R$. 
As in Proposition \ref{prop:fredholmreducestofinitedim} we can conclude that the
projection onto $\widetilde{H}_0\oplus\R$ restricts to a homeomorphism
$\tilde{F}^{-1}(0)\simeq(\pi_{\mathcal{O}}\circ\Phi)^{-1}(0)$. It is simple to
check that $\Phi$ is invariant under translations in $\R$.
Restricting $\Phi$ to $\mathcal{I}:=\widetilde{H}_0$ proves the first
claim.

Now let us further assume that $\boldsymbol{\mu}=2+\epsilon$ and that all
singularities are stable. Here, $\epsilon$ is to be understood as in Remark
\ref{rem:muregularity}; in particular, the moduli space we will obtain is
independent of the particular $\epsilon$ chosen. Recall from Corollary
\ref{cor:laplaceresults} that for $\boldsymbol{\mu}>2-m$ we can compute the
dimension of $\mbox{Coker}(\Delta_g)$ in terms of the number of harmonic
functions on the cones $\mathcal{C}_i$. Recall from Definition
\ref{def:stability} that SL cones always admit a certain number of harmonic
functions. This implies that, for the operator
$\Delta_g:W^p_{k,\boldsymbol{\mu}}(L)\rightarrow
W^p_{k-2,\boldsymbol{\mu}-2}(L)$, 
\begin{equation}
\mbox{dim(Coker$(\Delta_g)$)}\geq d, \ \ \mbox{where } d:=\sum_{i=1}^e
\left(1+2m+m^2-1-\mbox{dim}(G_i)\right).
\end{equation}
The stability condition is equivalent to $\mbox{dim(Coker$(\Delta_g)$)}=d$. This
means that the cokernel of the operator in Equation \ref{eq:cslaplacian} has
dimension $d-1$. Notice that $d$ is also the dimension of the space
$T_e\tilde{\E}\oplus E_0$. Our calculation of the kernel thus implies that the
map $d\tilde{F}[0]$ of Equation \ref{eq:csrestrictedlin} is surjective. Thus
$\mathcal{O}=\{0\}$. We can now apply the Implicit Function Theorem directly to
$\tilde{F}$ to obtain that $\tilde{F}^{-1}(0)$ is smooth, of dimension
$\mbox{dim}(\widetilde{H}_0)+1$. Quotienting by $\R$ and using Equation
\ref{eq:dimexactseq} gives the desired result.
\end{proof}
We call $\mathcal{O}$ the \textit{obstruction space} of the SL deformation
problem. 

\subsubsection*{CS/AC special Lagrangians in $\C^m$} We can now state and prove the main result of this paper.
\begin{theorem} \label{thm:accssl}
Let $L$ be a CS/AC SL submanifold of $\C^m$ with $s$ CS ends, $l$ AC ends
and rate $(\boldsymbol{\mu},\boldsymbol{\lambda})$. Let $\mathcal{M}_L$ denote
the moduli space of SL deformations of $L$ with moving singularities and rate
$(\boldsymbol{\mu},\boldsymbol{\lambda})$. Assume
$(\boldsymbol{\mu},\boldsymbol{\lambda})$ is non-exceptional for the map
\begin{equation}\label{eq:accslaplacian}
\Delta_g:W^p_{k,(\boldsymbol{\mu},\boldsymbol{\lambda})}(L)\rightarrow
W^p_{k-2,(\boldsymbol{\mu}-2,\boldsymbol{\lambda}-2)}(L).
\end{equation}
We will restrict our attention to the two cases $\boldsymbol{\lambda}\in
(2-m,0)$ or $\boldsymbol{\lambda}\in (0,2)$. In either case $\mathcal{M}_L$ is
locally homeomorphic to the zero set of a smooth map 
$\Phi:\mathcal{I}\rightarrow\mathcal{O}$ defined (locally) between
finite-dimensional vector spaces. If furthermore $\boldsymbol{\mu}=2+\epsilon$
and all singularities are stable then $\mathcal{O}=\{0\}$ and $\mathcal{M}_L$ is
smooth of dimension $\mbox{dim}(\mathcal{I})$. Specifically: 
\begin{enumerate}
\item If $\boldsymbol{\lambda}\in (2-m,0)$ then
$\mbox{dim}(\mathcal{I})=b^1_c(L)-s$.
\item If $\boldsymbol{\lambda}\in (0,2)$ then
$\mbox{dim}(\mathcal{I})=b^1_{c,\bullet}(L)-s+\sum_{i=1}^l d_i$,
where $d_i$ is the number of harmonic functions on the AC end $S_i$ of the form
$r^\gamma\sigma(\theta)$ with $\gamma\in [0,\lambda_i]$.
\end{enumerate}
\end{theorem}
\begin{proof}
Start with a map $F$ defined as in the previous theorems on $\tilde{\E}\times
W^p_{k-1,(\boldsymbol{\mu}-1,\boldsymbol{\lambda}-1)}(\Lambda^1)$, such that
$\mathcal{M}_L\simeq F^{-1}(0)$. Let
$\Delta_{\boldsymbol{\mu},\boldsymbol{\lambda}}$ denote the map of Equation
\ref{eq:accslaplacian}.

We split the proof into two parts, depending on the range of
$\boldsymbol{\lambda}$. To begin, assume $\boldsymbol{\lambda}\in (2-m,0)$. By
regularity and Decomposition \ref{decomp:closedforms_accs}, $\mathcal{M}_L$ is
locally homeomorphic to $\tilde{F}^{-1}(0)/\R$, where
$\tilde{F}$ is the (locally-defined) map
\begin{eqnarray*}
\tilde{F}:\tilde{\E}\times \widetilde{H}_{0,\infty}\times E_{0,\infty}\times
W^p_{k,(\boldsymbol{\mu},\boldsymbol{\lambda})}(L) &\rightarrow&
W^p_{k-2,(\boldsymbol{\mu}-2,\boldsymbol{\lambda}-2)}(L)\\
(\tilde{e},\beta,v,f) &\mapsto& F(\tilde{e},\beta+dv+df).
\end{eqnarray*}
As in Proposition \ref{prop:accsnonlinear},
$d\tilde{F}[0](y,\beta,v,f)=d^*\beta+\Delta_g(\chi(y)+v+f)$. Now consider the
restricted map
\begin{equation}\label{eq:accsrestrictedlin}
d\tilde{F}[0]:T_e\tilde{\E}\oplus E_0\oplus
W^p_{k,(\boldsymbol{\mu},\boldsymbol{\lambda})}(L)\rightarrow
W^p_{k-2,(\boldsymbol{\mu}-2,\boldsymbol{\lambda}-2)}(L),
\end{equation}
where $E_0$ is the subspace of functions in $E_{0,\infty}$ which vanish on the
AC ends.  
Notice that $\chi(y)+v+f\in
W^p_{k,(-\boldsymbol{\epsilon},\boldsymbol{\lambda})}(L)$. As in Theorem
\ref{thm:joyce} we can use Corollary \ref{cor:laplaceresults} to prove that the
map of Equation \ref{eq:accsrestrictedlin} is injective.

Let $\mathcal{O}$ denote the cokernel of the map of Equation
\ref{eq:accsrestrictedlin}. More precisely, we define it to be a
finite-dimensional subspace of
$W^p_{k-2,(\boldsymbol{\mu}-2,\boldsymbol{\lambda}-2)}(L)$ such that 
\begin{equation}
\mathcal{O}\oplus d\tilde{F}[0]\left(T_e\tilde{\E}\oplus E_0\oplus
W^p_{k,(\boldsymbol{\mu},\boldsymbol{\lambda})}\right)=W^p_{k-2,(\boldsymbol{\mu
}-2,\boldsymbol{\lambda}-2)}(L).
\end{equation}
Consider the map
\begin{eqnarray*}
G:\mathcal{O}\times \tilde{\E}\times \widetilde{H}_{0,\infty}\times
E_{0,\infty}\times W^p_{k,(\boldsymbol{\mu},\boldsymbol{\lambda})}(L)
&\rightarrow& W^p_{k-2,(\boldsymbol{\mu}-2,\boldsymbol{\lambda}-2)}(L)\\
(\gamma,\tilde{e},\beta,v,f) &\mapsto& \gamma+\tilde{F}(\tilde{e},\beta,v,f).
\end{eqnarray*}
By construction the restriction of $dG[0]$ to the space $\mathcal{O}\oplus
T_e\tilde{\E}\oplus E_0\oplus W^p_{k,(\boldsymbol{\mu},\boldsymbol{\lambda})}$
is an isomorphism. Let $E'$ denote a complement of $E_0\oplus\R$ in
$E_{0,\infty}$, \textit{i.e.} $E_{0,\infty}=E_0\oplus\R\oplus E'$. As in Theorem
\ref{thm:joyce}, $G^{-1}(0)$ is smooth and can be written as the graph of a
smooth map $\Phi$ defined on $\widetilde{H}_{0,\infty}\oplus(\R\oplus E')$.
Restricting $\Phi$ to $\mathcal{I}:=\widetilde{H}_{0,\infty}\oplus E'$ and using
the same arguments as in Proposition \ref{prop:fredholmreducestofinitedim} and
Theorem \ref{thm:joyce} then proves
the first claim regarding $\mathcal{M}_L$ for this range of
$\boldsymbol{\lambda}$. Notice that
$\mbox{dim}(\widetilde{H}_{0,\infty})=b^1_c(L)-(s+l)+1$ and $\mbox{dim}(E')=l-1$
so $\mbox{dim}(\mathcal{I})=b^1_c(L)-s$.

Now let us further assume that $\boldsymbol{\mu}=2+\epsilon$ and that all
singularities are stable. Here, as in Theorem \ref{thm:joyce}, $\epsilon$ is to
be understood as in Remark \ref{rem:muregularity}. By Corollary
\ref{cor:laplaceresults} and the definition of stability,
\begin{equation}
\mbox{dim(Coker$(\Delta_{\boldsymbol{\mu},\boldsymbol{\lambda}})$)}=d, \ \
\mbox{where } d:=\sum_{i=1}^s \left(1+2m+m^2-1-\mbox{dim}(G_i)\right).
\end{equation}
Again, $d$ is also the dimension of the space $T_e\tilde{\E}\oplus E_0$. Our
previous injectivity calculation thus implies that the map $d\tilde{F}[0]$ of
Equation \ref{eq:accsrestrictedlin} is an isomorphism. In particular,
$\mathcal{O}=\{0\}$. We can now apply the Implicit Function Theorem directly to
$\tilde{F}$ to obtain that $\tilde{F}^{-1}(0)$ is smooth. Quotienting by $\R$
shows that $\mathcal{M}_L$ is smooth.

We now start over again, under the assumption $\boldsymbol{\lambda}\in (0,2)$.
In this case we use the map
\begin{eqnarray*}
\tilde{F}:\tilde{\E}\times \widetilde{H}_{0,\bullet}\times E_0\times
W^p_{k,(\boldsymbol{\mu},\boldsymbol{\lambda})}(L) &\rightarrow&
W^p_{k-2,(\boldsymbol{\mu}-2,\boldsymbol{\lambda}-2)}(L)\\
(\tilde{e},\beta,v,f) &\mapsto& F(\tilde{e},\beta+dv+df)
\end{eqnarray*}
and the restricted map
\begin{equation}\label{eq:accsrestrictedlinbis}
d\tilde{F}[0]:T_e\tilde{\E}\oplus E_0\oplus
W^p_{k,(\boldsymbol{\mu},\boldsymbol{\lambda})}(L)\rightarrow
W^p_{k-2,(\boldsymbol{\mu}-2,\boldsymbol{\lambda}-2)}(L).
\end{equation}
Recall the construction of $E_0$ in Decomposition \ref{decomp:closedforms_accs}:
it is clear that we may assume that $\chi(T_e\tilde{\E})$ and $E_0$ are linearly
independent in $W^p_{k,(-\boldsymbol{\epsilon},-\boldsymbol{\epsilon})}(L)$.
Corollary \ref{cor:laplaceresults} proves that $\Delta_g$ is injective on this
space. Define a decomposition
\begin{equation}
T_e\tilde{\E}\oplus E_0=Z'\oplus Z''
\end{equation}
by imposing $\Delta_g(Z')=\Delta_g(T_e\tilde{\E}\oplus E_0)\cap
\mbox{Im}(\Delta_{\boldsymbol{\mu},\boldsymbol{\lambda}})$ and choosing any
complement $Z''$. Then one can check that the kernel of the map of Equation
\ref{eq:accsrestrictedlinbis} is isomorphic to
$Z'\oplus\mbox{Ker}(\Delta_{\boldsymbol{\mu},\boldsymbol{\lambda}})$.

Choose $\mathcal{O}$ in
$W^p_{k-2,(\boldsymbol{\mu}-2,\boldsymbol{\lambda}-2)}(L)$ such that 
\begin{equation}
\mathcal{O}\oplus d\tilde{F}[0]\left(T_e\tilde{\E}\oplus E_0\oplus
W^p_{k,(\boldsymbol{\mu},\boldsymbol{\lambda})}\right)=W^p_{k-2,(\boldsymbol{\mu
}-2,\boldsymbol{\lambda}-2)}(L).
\end{equation}
Consider the map
\begin{eqnarray*}
G:\mathcal{O}\times \tilde{\E}\times \widetilde{H}_{0,\bullet}\times E_0\times
W^p_{k,(\boldsymbol{\mu},\boldsymbol{\lambda})}(L) &\rightarrow&
W^p_{k-2,(\boldsymbol{\mu}-2,\boldsymbol{\lambda}-2)}(L)\\
(\gamma,\tilde{e},\beta,v,f) &\mapsto& \gamma+\tilde{F}(\tilde{e},\beta,v,f).
\end{eqnarray*}
The restriction of $dG[0]$ to the space $\mathcal{O}\oplus T_e\tilde{\E}\oplus
E_0\oplus W^p_{k,(\boldsymbol{\mu},\boldsymbol{\lambda})}$ is surjective. As
before, this implies that $G^{-1}(0)$ can be parametrised via a smooth map
$\Phi$ defined (locally) on the space $\widetilde{H}_{0,\bullet}\oplus Z'\oplus
\mbox{Ker}(\Delta_{\boldsymbol{\mu},\boldsymbol{\lambda}})$. As usual, these
maps are invariant under translations in $\R\subset
Z'\oplus\mbox{Ker}(\Delta_{\boldsymbol{\mu},\boldsymbol{\lambda}})$. Setting
$\mathcal{I}:=(\widetilde{H}_{0,\bullet}\oplus Z'\oplus
\mbox{Ker}(\Delta_{\boldsymbol{\mu},\boldsymbol{\lambda}}))/\R$ and considering
the natural map on this quotient then proves the first claim regarding
$\mathcal{M}_L$ for this range of $\boldsymbol{\lambda}$.

Now assume that $\boldsymbol{\mu}=2+\epsilon$ and that all singularities are
stable. Choose $\boldsymbol{\lambda}'\in (2-m,0)$. We can restrict the map of
Equation \ref{eq:accsrestrictedlinbis} to the map
\begin{equation}\label{eq:accsrestrictedlinter}
d\tilde{F}[0]:T_e\tilde{\E}\oplus E_0\oplus
W^p_{k,(\boldsymbol{\mu},\boldsymbol{\lambda}')}(L)\rightarrow
W^p_{k-2,(\boldsymbol{\mu}-2,\boldsymbol{\lambda}'-2)}(L).
\end{equation}
Exactly as for Equation \ref{eq:accsrestrictedlin}, it is simple to prove that
Equation \ref{eq:accsrestrictedlinter} defines an isomorphism and that
$\mbox{dim}(T_e\tilde{\E}\oplus
E_0)=\mbox{dim(Coker($\Delta_{\boldsymbol{\mu},\boldsymbol{\lambda}'}$))}$,
where 
\begin{equation*}
\Delta_{\boldsymbol{\mu},\boldsymbol{\lambda}'}:=\Delta_g:
W^p_{k,(\boldsymbol{\mu},\boldsymbol{\lambda}')}(L)\rightarrow
W^p_{k-2,(\boldsymbol{\mu}-2,\boldsymbol{\lambda}'-2)}(L).
\end{equation*}
One can check that the dimension of
$\mbox{Coker}(\Delta_{\boldsymbol{\mu},\boldsymbol{\lambda}})$ decreases as
$\boldsymbol{\lambda}$ increases. We can actually assume, cf.
\cite{pacini:weighted}, that
$\mbox{Coker}(\Delta_{\boldsymbol{\mu},\boldsymbol{\lambda}})\subseteq\mbox{
Coker}(\Delta_{\boldsymbol{\mu},\boldsymbol{\lambda}'})$. This proves that the
map of Equation \ref{eq:accsrestrictedlinbis} is surjective, \textit{i.e.}
$\mathcal{O}=\{0\}$, so $\tilde{F}^{-1}(0)$ and $\mathcal{M}_L$ are smooth. To
compute the dimension of this moduli space notice that $Z''\simeq
\mbox{Coker}(\Delta_{\boldsymbol{\mu},\boldsymbol{\lambda}})$ so
\begin{eqnarray}
\mbox{dim(Ker$(d\tilde{F}[0])$)} &=&
\mbox{dim(Ker$(\Delta_{\boldsymbol{\mu},\boldsymbol{\lambda}})$)}
+\mbox{dim}(Z')\nonumber\\
&=&
\mbox{dim(Ker$(\Delta_{\boldsymbol{\mu},\boldsymbol{\lambda}})$)}+\mbox{
dim(Coker($\Delta_{\boldsymbol{\mu},\boldsymbol{\lambda}'}$))}-\mbox{
dim(Coker($\Delta_{\boldsymbol{\mu},\boldsymbol{\lambda}}$))}\nonumber\\
&=&
i(\Delta_{\boldsymbol{\mu},\boldsymbol{\lambda}})-i(\Delta_{\boldsymbol{\mu},
\boldsymbol{\lambda}'}),
\end{eqnarray}
where $i$ denotes the index of the Fredholm map. This implies that the kernel of
the full map $d\tilde{F}[0]$ has dimension
$\mbox{dim}(\widetilde{H}_{0,\bullet})+i(\Delta_{\boldsymbol{\mu},\boldsymbol{
\lambda}})-i(\Delta_{\boldsymbol{\mu},\boldsymbol{\lambda}'})$. The conclusion
follows from Equation \ref{eq:cohomsequence_accsbis} and the change of index
formula, cf. \cite{pacini:weighted}.
\end{proof}

\begin{remark}\label{rem:coh_dim}
Notice that, when $\boldsymbol{\lambda}<0$ and the stability condition is
verified, the dimension of the SL moduli spaces appearing in Theorems
\ref{thm:mclean}, \ref{thm:pacini}, \ref{thm:joyce} and \ref{thm:accssl} is
purely topological. The cases analyzed in the theorems correspond exactly to the
cases analyzed in Corollary \ref{cor:topsummary}, in the sense that the moduli
spaces should be thought of as being modelled on the cohomology spaces which
appear in Corollary \ref{cor:topsummary}. 

It is interesting to notice how decay conditions on AC and CS ends are
incorporated differently into these cohomology spaces: decay conditions on AC
ends correspond to using compactly-supported forms while decay conditions on CS
ends correspond to the condition that a certain restriction map vanishes, cf. also Remark \ref{rem:topsummary}.

Allowing $\boldsymbol{\lambda}>0$ changes the topological data, again in
agreement with Corollary \ref{cor:topsummary}. It also introduces new SL
deformations which depend on analytic data.
\end{remark}

\begin{example}
Let $\mathcal{C}$ be a SL cone in $\C^m$. Assume $\mathcal{C}$ is stable and
that its link $\Sigma$ is connected so that $s=1$. Using Poincar\'{e} Duality
and the fact that $\mathcal{C}\simeq\Sigma\times (0,\infty)$ we see that 
\begin{equation}
b^1_c(\mathcal{C})=b^{m-1}(\mathcal{C})=b^{m-1}(\Sigma)=1.
\end{equation}
Theorem \ref{thm:accssl} then shows that, for $\lambda\in (2-m,0)$,
$\mathcal{M}_{\mathcal{C}}$ has dimension 0, \textit{i.e.} $\mathcal{C}$ is
rigid within this class of deformations.

Notice also that restriction defines isomorphisms $H^i(\mathcal{C};\R)\simeq
H^i(\Sigma;\R)$ so the long exact sequence \ref{eq:cohomsequence_accsbis}, using
$\Sigma_0=\Sigma$, leads to $H^i_{c,\bullet}(\mathcal{C};\R)=0$. 
Theorem \ref{thm:accssl} then shows that $\mathcal{M}_{\mathcal{C}}$ has
dimension 0 if $\lambda\in (0,1)$ and has dimension $2m$ if $\lambda\in (1,2)$.
In the latter case the SL deformations are simply the translations of
$\mathcal{C}$ in $\C^m$.
\end{example}


\ 

{\bf Acknowledgments.} I would like to thank D. Joyce for many useful
explanations on his work and for his help and advice on various parts of this
paper. I would also like to thank S. Karigiannis and Y. Song for useful conversations. The main part of this 
work was carried out while I was a Marie Curie EIF Fellow at the University of
Oxford.

\bibliographystyle{amsplain}
\bibliography{accssldefs}
\end{document}